\documentclass[11pt,oneside,english,jou, a4paper]{amsart}
\usepackage{a4wide}
\usepackage[T1]{fontenc}
\usepackage[latin9]{inputenc}
\usepackage{babel}
\usepackage{verbatim}
\usepackage{mathrsfs}
\usepackage{amstext}
\usepackage{amsthm}
\usepackage{amssymb}
\usepackage{esint}
\usepackage{enumerate}
\usepackage[margin=2.3cm]{geometry}
\usepackage[title]{appendix}

\makeatletter

\theoremstyle{plain}
\newtheorem{thm}{\protect\theoremname}
\theoremstyle{remark}
\newtheorem{rem}[thm]{\protect\remarkname}
\theoremstyle{definition}
\newtheorem{defn}[thm]{\protect\definitionname}
\theoremstyle{plain}
\newtheorem{prop}[thm]{\protect\propositionname}
\theoremstyle{plain}
\newtheorem{lem}[thm]{\protect\lemmaname}
\theoremstyle{plain}

\theoremstyle{plain}
\newtheorem{cor}[thm]{\protect\corollaryname}
\theoremstyle{definition}
\newtheorem{problem}[thm]{\protect\problemname}
\theoremstyle{definition}
\newtheorem{examplex}[thm]{\protect\examplename}
\newenvironment{example}
{\pushQED{\qed}\examplex}
{\popQED\endexamplex}

\numberwithin{thm}{subsection}

\usepackage{tikz}
\usetikzlibrary{shapes,arrows}

\newcommand{\eps}{\varepsilon}
\newcommand{\R}{\mathbb{R}}
\newcommand{\N}{\mathbb{N}}
\newcommand{\Z}{\mathbb{Z}}
\newcommand{\mH}{\mathcal{H}}\newcommand{\mA}{\mathcal{A}}\newcommand{\mB}{\mathcal{B}}\newcommand{\mC}{\mathcal{C}}\newcommand{\mD}{\mathcal{D}}\newcommand{\mS}{\mathcal{S}}

\renewcommand{\r}{\mathrm{r}}

\newcommand{\E}{\mathbb{E}}
\newcommand{\mX}{\mathcal{X}}
\newcommand{\mP}{\mathcal{P}}
\newcommand{\bp}{\begin{proof}}
	\newcommand{\ep}{\end{proof}}
\newcommand{\ummdim}{\overline{\mathrm{mdim}}_M}
\newcommand{\lmmdim}{\underline{\mathrm{mdim}}_M}
\newcommand{\mmdim}{\mathrm{mdim}_M}
\newcommand{\udim}{\overline{\mathrm{dim}}_B}
\newcommand{\ldim}{\underline{\mathrm{dim}}_B}
\newcommand{\umbdim}{\overline{\mathrm{mdim}_B}}
\newcommand{\mbdim}{\overline{\mathrm{mdim}}_B}

\newcommand{\uid}{\overline{\mathrm{ID}}}

\newcommand{\diam}{\mathrm{diam}}
\newcommand{\mdim}{\mathrm{mdim}}

\newcommand{\id}{\mathrm{id}}

\newcommand{\supp}{\mathrm{supp}}
\newcommand{\Prob}{\mathrm{Prob}}
\newcommand{\dist}{\mathrm{dist}}

\newcommand{\Ker}{\mathrm{Ker}}

\newcommand{\LIN}{\mathrm{LIN}}

\makeatother

\providecommand{\conjecturename}{Conjecture}
\providecommand{\corollaryname}{Corollary}
\providecommand{\definitionname}{Definition}
\providecommand{\examplename}{Example}
\providecommand{\lemmaname}{Lemma}
\providecommand{\problemname}{Problem}
\providecommand{\propositionname}{Proposition}
\providecommand{\remarkname}{Remark}
\providecommand{\theoremname}{Theorem}

\begin{document}
\sloppy
\title{Metric mean dimension and analog compression}

\author[Y. Gutman]{Yonatan Gutman}
\address{Y. Gutman: Institute of Mathematics, Polish Academy of Sciences,
	ul. \'Sniadeckich 8, 00-656 Warszawa, Poland}
\email{y.gutman@impan.pl}

\author[A. \'{S}piewak]{Adam \'{S}piewak}
\address{A. \'{S}piewak: Institute of Mathematics, University of Warsaw, ul. Banacha 2, 02-097 Warszawa, Poland}
\email{a.spiewak@mimuw.edu.pl}

\begin{abstract}
Wu and Verd\'{u} developed a theory of almost lossless analog compression,
where one imposes various regularity conditions on the compressor
and the decompressor with the input signal being modelled by a (typically
infinite-entropy) stationary stochastic process. In this work we consider all stationary stochastic processes with trajectories in a prescribed set of (bi-)infinite sequences and find uniform lower and upper bounds for certain compression rates in terms of \textit{metric mean dimension} and \textit{mean box dimension}. An essential tool is the recent Lindenstrauss-Tsukamoto variational principle expressing metric mean dimension in terms of rate-distortion functions. We obtain also lower bounds on compression rates for a fixed stationary process in terms of the rate-distortion dimension rates and study several examples.
\end{abstract}


\thanks{This paper was presented in part at 2019 IEEE International Symposium on Information Theory (Paris). See \cite{GS19short} for a conference proceedings version.}
\thanks{Y.G was partially supported by the National Science Center (Poland) Grant 2013/08/A/ST1/00275. Y.G and A.\'S were partially supported by the National Science Center (Poland) grant 2016/22/E/ST1/00448.}
\maketitle

\section{Introduction}\label{sec:intro}
\subsection{Overview}
In recent years, the theory of compression for analog sources (modeled by real-valued stochastic processes) underwent a major development (as a sample of such results see \cite{WV10}, \cite{WV12}, \cite{JP16}, \cite{jalali2017universal}, \cite{CompressionCS17}, \cite{stotz2013almost}, \cite{SRAB17}, \cite{GK18}). There are two key differences with the classical Shannon's model of compression for discrete sources. The first one is the necessity to employ regularity conditions on the compressor and/or decompressor functions (e.g. linearity or Lipschitz/H\"{o}lder continuity). This requirement makes the problem both non-trivial (since highly irregular bijections between $\R$ and $\R^n$ cannot be applied to obtain arbitrary small compression rates) and reasonable from the point of view of applications (since it induces robustness to noise). The second difference is the fact that non-discrete sources have in general infinite Shannon entropy rate, hence a different measure of complexity for stochastic processes must be considered. One of the most fruitful approaches taken in the literature is to assume a specific structure of the source signal. Apparently the most prominent instance of such an approach is the theory of compressed sensing, where the input vectors are assumed to be sparse. \footnotemark \footnotetext{cf. \cite{CT05}, \cite{CRT06}, \cite{Donoho06}. For a comprehensive treatment on compressed sensing see \cite{FR13}. For applications in sampling theory see \cite{eldar2015sampling}. For similar results in the more general case of signals from unions of linear subspaces see \cite{BD09} and  \cite{Blumensath11}.} In this setting, the theory of linear compression with efficient and stable recovery algorithms has been developed. However, strong assumptions posed on the structure of the signal reduce the applicability of the technique. A different approach was developed in the pioneering work of Wu and Verd\'u \cite{WV10}. Instead of making assumptions on the structure of the signal, new measures of complexity related to the box-counting (Minkowski) dimension of the signal were introduced and proved to be bounds on the compression rates for certain classes of compressors and decompressors. Jalali and Poor (\cite{JP16}, \cite{jalali2017universal}) developed the theory of universal compressed sensing for stochastic processes with the $\psi^*$-mixing property, where the compression and decompression algorithms do not require an a priori knowledge of the source distribution. The corresponding compression rates are given in terms of a certain generalization of the R\'{e}nyi information dimension. Compression algorithms for general stationary processes (assuming an a priori knowledge of process dependent compression codes) were obtained in \cite{CompressionCS17}.

The goal of this paper is twofold. We adapt the setting from \cite{WV10}, but instead of a single process we consider all stationary stochastic processes with trajectories in a prescribed set $\mS \subset [0,1]^\Z$ of (bi-)infinite sequences. This corresponds to an a priori knowledge of all possible trajectories of the process rather than its distribution. We are interested in the question of calculating the almost lossless compression rate in the sense of \cite{WV10} sufficient for all stationary stochastic processes supported in $\mS$ (i.e. taking into account the worst-case scenario) for compressors and decompressors in given regularity classes. Such a rate is universal in the sense that it can be achieved for every stationary stochastic process, independently of its distribution, as long as it fulfils the geometric constraints given by $\mS$. Universality, especially from the algorithmic point of view, is an important subject in compression theory - in the discrete case, the classical Lempel-Ziv codes \cite{LempelZiv78} attain the optimal compression rate without knowledge of the distribution of the source. Our method can be applied whenever the distribution governing the source is not known, but geometric information is given (e.g. signals emitted by the source are known to be sparse, but the corresponding probabilities are not available). We use the geometric properties of the system $\mS$ to obtain information on compression rates of the stochastic processes supported in $\mS$. Our approach can be seen as a combination of deterministic and stochastic perspectives on signal compression. Such an approach was already hinted in the classical work of Ziv \cite{Z78}, where coding theorems for individual sequences were studied and related to the Shannon's entropy theory dealing with probabilistic sources (see also \cite{KawanYuksel18, KawanYukselStoch18, LinderZamir06}).

We focus on two compression schemes: with Borel or linear compressors and H\"{o}lder (Lipschitz) decompressors.  Such or similar schemes were considered in the literature in the case of a single stochastic process (e.g. \cite{WV10, jalali2017universal, CompressionCS17, SRAB17}). Except for considering an ensemble of processes instead of a single one, there are two main differences in our approach with respect to \cite{WV10}. We consider processes with trajectories in $[0,1]^\Z$, hence we assume a uniform bound on the attained values. We also demand decompressor functions to fulfil $(L, \alpha)$-H\"{o}lder condition with fixed $L>0$ and $\alpha \in (0,1]$ for all block lengths. In particular, the Lipschitz constant $L$ has to be uniformly bounded for arbitrary long codes. This novelty guarantees improved robustness to noise. Our results are upper and lower bounds for such rates in terms of certain dynamical-geometric characteristics of the considered set $\mS$. The main result is a lower bound on the worst-case compression rate for Borel compression and $(L,\alpha)$-H\"{o}lder decompression in terms of the \emph{metric mean dimension} of the set $\mS$. We illustrate this fundamental limit by several examples, including parametrized family of sparse signal subshifts, where we show that our fundamental bound turns out to be tight (in the asymptotic range $L \to \infty$). Admitting general Borel compressors makes this result applicable to a wide class of compression schemes (e.g. both linear and ones involving quantization). This result is accompanied by two upper bounds (achievability results). The first one is an upper bound in the context of Borel compression and $(L,\alpha)$-H\"{o}lder decompression. It is obtained by employing constructions akin to the Peano space-filling curve of H\"{o}lder surjective maps between unit cubes. This type of construction would not be possible for more stringent regularity requirements for the compressor. The second upper bound is a zero-probability error bound for linear compression with $(L,\alpha)$-H\"{o}lder decompression. It is given in terms of the \emph{mean box dimension} of the set $S$ and it is deduced from a finite-dimensional embedding theorem involving the upper box-counting (Minkowski) dimension (see \cite{benartzi1993holder}). We obtain also lower bounds on compression rates for a fixed stationary process in terms of the rate-distortion dimension (see \cite{CompressionCS17}; see also \cite{GK18} for a more detailed treatment on the connections between various notions of dimension for stochastic processes).

The fact that above bounds are obtained in terms of the metric mean dimension and mean box dimension of the set $\mS$ realizes the second goal of the paper: introducing notions from the theory of dynamical systems to the study of compression rates. As we consider stationary processes, it is natural to assume the set $\mS$ to be invariant under the shift transformation and hence it can be considered as a topological dynamical system. Metric mean dimension is a geometrical invariant of dynamical systems introduced and studied by Lindenstrauss and Weiss in \cite{LW00}. Existence of connections between signal processing and mean dimension theory was observed first in \cite{GutTsu16}, where the use of sampling theorems was an essential tool for proving the embedding conjecture of Lindenstrauss (cf. \cite{L99}, see also \cite{gutman2017application}). Another connection between these domains was established recently in \cite{lindenstrauss_tsukamoto2017rate}, where a variational principle for metric mean dimension was given in terms of rate-distortion functions - objects well studied in information theory, describing the entropy rate of the quantization of a stochastic process at a given scale (for more details see \cite{Gray11}). This variational principle is our main tool in developing lower bounds on compression rates for all stationary processes supported in $\mS$. Mean box dimension is a box-counting dimension version of the projective dimension introduced by Gromov in \cite{G}. The results of this paper can be also seen as an attempt to provide operational meaning for metric mean dimension in stochastic terms (in the sense that it gives fundamental limits for the corresponding compression problem), hence they contribute to the program of relating ergodic theory to the theory of mean dimension initiated in \cite{GutTsu16}.

\subsection{Organization of the paper}

In Section \ref{sec:notation} we introduce basic notions and notations. In Section \ref{sec:analog_compression_rates} we define analog compression rates (following \cite{WV10}). In Section \ref{sec:dim} we introduce geometric and dynamical notions of dimension (box-counting dimension, metric mean dimension and mean box dimension) and illustrate them by basic examples. In Section \ref{sec:main} we state the main results of the paper. In Section \ref{sec:examples} we present several examples, which serve as illustrations to the main results. Section \ref{sec:rate_dist} contains the variational principle for metric mean dimension in the case of subshifts, which is our main tool in developing lower bounds on certain compression rates. Section \ref{sec:related_works} discusses related works by other authors. Section \ref{sec:proofs} contains proofs of the main results. Section \ref{sec:concluding_remarks} contains concluding remarks. In the Appendices we give more details on metric mean dimension theory from the perspective of dynamical systems, prove several auxiliary results and present calculations for examples from Sections \ref{sec:dim} and \ref{sec:examples}.
\numberwithin{thm}{section}
\section{Notation}\label{sec:notation}
By $\N = \{1, 2, \ldots\}$ we will denote the set of natural numbers, by $\N_0 = \{0, 1, 2, \ldots\}$ the set of natural numbers with zero and by $\Z$ the set of integers. We take the base of the logarithm to be equal to $2$. For a (finite or infinite) vector $x = (x_n)$ and $\ell \leq k$ we denote $x|_{\ell}^k = (x_{\ell}, \ldots, x_k)$. By $\#A$ we will denote the cardinality of the set $A$.

In this paper, we apply results from the theory of dynamical systems to the theory of signal processing. From the signal processing perspective, we will consider stationary stochastic process $\{ X_n : \Omega \to [0,1]\ |\ n \in \Z \}$ defined on some probability space $(\Omega, \mathbb{P})$. Usually, instead of a single process, we will be interested in considering all the stationary processes with trajectories in some prescribed set $\mS$.

A useful description of this setting can be given in terms of dynamical systems and we will usually take this perspective. By a \textbf{(topological) dynamical system} we will understand a triple $(\mX, T, \rho)$, where $(\mX, \rho)$ is a compact metric space with metric $\rho$ and $T: \mX \to \mX$ is a homeomorphism. Many notions which we consider (e.g. metric mean dimension) are metric dependent, hence we do not suppress $\rho$ from the notation. For a (countably-additive) Borel measure $\mu$ on $\mX$, by $T_* \mu$ we will denote its push-forward by $T$, i.e. a Borel measure on $\mX$ given by $T_*\mu(A) = \mu(T^{-1}(A))$ for Borel $A \subset \mX$. We will say that measure $\mu$ is \textbf{$T$-invariant}, if $T_* \mu = \mu$. By $\Prob(\mX)$ we will denote the set of all Borel probability measures on $\mX$ and by $\mP_{T}(\mX)$ the set of all $T$-invariant Borel probability measures on $\mX$. For a Borel measure $\mu$ on $\mX$ its \textbf{support} is the closed set $\supp(\mu) = \{ x \in \mX : \mu(U) > 0 \text{ for all open neighbourhoods U of } x   \}.$ This is the smallest closed set of full measure $\mu$. For an introduction to topological dynamics and its connections with ergodic theory see \cite[Chapters 5-8]{W82}.

Formally, the connection between signal processing and dynamical systems is given by considering a special class of systems: subshifts with shift transformation. Assume that $(A,d)$ is a compact metric space, which we will treat as the alphabet space. We will focus on the case $A\subset[0,1]$ with the standard metric. $A^{\Z}$ is itself a compact metrizable space when endowed with the product topology. This topology is metrizable by the product metric
\begin{equation}\label{e:product_metric} \rho(x,y) = \sum \limits_{i \in \Z} \frac{d(x_i, y_i)}{2^{|i|}},
\end{equation}
where $x = (x_i)_{i \in \Z},\ y = (y_i)_{i \in \Z}$. Note that $\diam(A^\Z, \rho) = 3\diam(A,d)$.
Define the \textbf{shift} \textbf{transformation} $\sigma: A^\Z \to A^\Z$ as \[\sigma((x_{i})_{i \in \Z})=(x_{i+1})_{i \in \Z}.\]
We will be interested in the properties of a given \textbf{subshift}, i.e. a closed (in the product topology) and shift-invariant subset $\mS \subset A^{\Z}$, which we will interpret as the set of all admissible signals that can occur in the input and treat as a dynamical system with transformation $\sigma$. For $n \in \N$ denote by $\pi_{n}:\mS\to A^{n}$ the projection\[ \pi_n(x) = x|_{0}^{n-1},\ x = (x_i)_{i \in \Z} \in \mS. \]
We will endow $[0,1]^\Z$ with a product metric as in (\ref{e:product_metric}), i.e. \begin{equation}\label{eq:tau} \tau(x,y)=\sum \limits_{i=-\infty}^{\infty}\frac{1}{2^{|i|}}|x_{i}-y_{i}|.\end{equation}
Throughout the paper, the letter $\tau$ will always denote this specific metric. It metrizes the product topology on $[0,1]^\Z$. Note that $\diam(\tau)=3$. This choice of the metric may seem arbitrary, but it turns out that metric mean dimension for subshifts takes a natural form when calculated with respect to $\tau$ (see Proposition \ref{prop:canonical}). We will consider mainly subshifts $\mS \subset [0,1]^\Z$, i.e. shift-invariant subsets of $[0,1]^\Z$ which are closed in the topology given by the metric $\tau$. In this language, for a given subshift $\mS \subset [0,1]^\Z$, the set $\mP_{\sigma}(\mS)$ consists of all distributions of stationary stochastic processes taking values in a given subshift $\mS$ (i.e. distributions of stationary stochastic processes $X_n,n \in \Z$ such that $(X_n)_{n \in \Z} \in \mS$ with probability one). For vectors $x, y \in [0,1]^n, x = (x_0, \ldots, x_{n-1}), y = (y_0, \ldots, y_{n-1})$ and $p \in [1, \infty)$ we define the (normalized) \textbf{$\ell^p$ distance} (coming from the $p$-th norm):
\[ \| x - y\|_p=\Big(\frac{1}{n}\sum_{k=0}^{n-1}|x_{k}-y_{k}|^p\Big)^{\frac{1}{p}}  \]
and
\[ \| x - y\|_\infty = \max \{ |x_k - y_k| : 1 \leq k \leq n \}. \]
The following inequalities hold on $[0,1]^n$ for $1 \leq p \leq q < \infty$ (see \cite[Chapter 3 - Exercise 5]{R87})
\[ \| \cdot \|_p \leq \| \cdot\|_q \leq \| \cdot \|_{\infty} \text{ and } \|\cdot\|_{\infty} \leq n^{\frac{1}{p}} \| \cdot \|_p.  \]
Note also that $\diam([0,1]^n, \| \cdot\|_p) = 1$ for every $n \in \N,\ p \in[1, \infty]$. For $x \in [0,1]^n$ we will denote $\supp(x) = \{ 0 \leq i < n : x_i \neq 0 \}$ and $\| x \|_0 = \#\supp(x)$.

Let $\nu$ be a Borel measure on $[0,1]$. We denote the product measure
induced by $\nu$ on $[0,1]^{\mathbb{{Z}}}$ as $\bigotimes\limits_{\mathbb{{Z}}}\nu$. It corresponds distributions of memoryless (i.i.d) sources.

A sequence $\N \ni n \mapsto a_n \in [0, \infty)$ is called \textbf{subadditive} if $a_{n+k} \leq a_n + a_k$ for every $n,k \in \N$. We will apply several times the following well-known fact (see e.g \cite[Propoition 6.2.3]{coornaert2015topological}): if $n \mapsto a_n$ is a subadditive sequence, the limit $\lim \limits_{n \to \infty} \frac{a_n}{n}$ exists and equals $\inf \limits_{n \in \N} \frac{a_n}{n}$ .

\begin{rem}\label{rem:R^N to =00005B0,1=00005D^Z} In the paper \cite{WV10}, Wu and Verd\'{u} consider processes taking values in the space $\R^\N$, whereas we will consider the space $[0,1]^\Z$. From the formal point of view, this assumption is made in order to obtain compactness of the considered space. From the point of view of applications, this corresponds to an a priori bound on signals (up to rescaling). We also consider two-sided sequences instead of one-sided - this corresponds to having an access to the past of the state of the system and makes the acting transformation $\sigma$ invertible, which is a useful assumption in dynamical systems. However, all notions  under consideration depend only on the future and the corresponding results remain unchanged in the non-invertible case.
\end{rem}
\numberwithin{thm}{subsection}
\section{Analog compression rates}\label{sec:analog_compression_rates}

In this section we introduce analog compression rates for sources with non-discrete alphabet. They measure the rate at which a given signal can be compressed assuming certain restrictions on the performance and regularity of the compression-decompression process. The definitions are analogous to the \emph{achievable rates} as defined in \cite[Definition 5]{WV10}. We consider subshifts in $[0,1]^{\Z}$ and invariant measures supported in them. In this setting it is natural to assume further constraints on the compressor and decompressor functions. This follows from the fact that the alphabet $[0,1]$ is infinite: there exists an injection from any $A \subset [0,1]^n$ into $[0,1]$. Therefore, if we do not assume any regularity of the compressor or decompressor functions, the asymptotic compression rate is always $0$. This is also true if we allow only Borel compressors and decompressors (see \cite[Section IV-B]{WV10}). On the other hand, from the point of view of applications it is also desirable to assume some regularity conditions on the compressor and decompressor functions, since they induce robustness to noise and enable numerical control of the errors occurring in the compression and decompression processes (e.g. by requiring encoders and decoders to be H\"{o}lder).

\subsection{Definitions of compression rates}\label{sec:compression_rates}
We will introduce now several regularity classes which will be used throughout the paper.
\begin{defn}
	\label{def:(regularity-classes)}
	A regularity class with respect to the norm $\| \cdot \|_p,\ p \in [1,\infty]$ is a set $\mathcal{C}$ of functions between finite dimensional unit cubes, i.e. $\mathcal{C} \subset \{ f : [0,1]^n \to [0,1]^k : n,k \in \N \}$. We will consider the following regularity classes:
	
	\begin{enumerate}
		\item \textbf{Borel} $\ \mathcal{B} = \{ f : [0,1]^n \to [0,1]^k \ \big|\ n,k \in \N,\ f \text{ is Borel}\}$
		\item \textbf{$L-$Lipschitz} for $L>0$
		\[\hfill \mathcal{L}_{L}^p = \{ f : [0,1]^n \to [0,1]^k \ \big|\ n,k \in \N,\ \|f(x)-f(y)\|_p\leq L\|x-y\|_p \text{ for all }\ x,y \in [0,1]^n \}\]
		\item \textbf{Lipschitz} \[ \mathcal{{L}} = \{ f : [0,1]^n \to [0,1]^k \ \big|\ n,k \in \N,\ \underset{L>0}{\exists}\ \|f(x)-f(y)\|_p\leq L\|x-y\|_p \text{ for all }\ x,y \in [0,1]^n \}\]
		\item \textbf{$(L,\alpha)$-H\"older} for $L > 0,\ \alpha \in (0,1]$ \[\mathcal{{H}}_{L,\alpha}^p = \{ f : [0,1]^n \to [0,1]^k \ \big|\ n,k \in \N,\ \|f(x)-f(y)\|_p\leq L\|x-y\|^{\alpha}_{p} \text{ for all }\ x,y \in [0,1]^n \}\]
		\item \textbf{$\alpha$-H\"older} for $\alpha \in (0,1]$ \[\mathcal{{H}}_{\alpha} = \{ f : [0,1]^n \to [0,1]^k \ \big|\ n,k \in \N,\ \underset{L>0}{\exists}\ \|f(x)-f(y)\|_p\leq L\|x-y\|_p^{\alpha} \text{ for all }\ x,y \in [0,1]^n \}\]
		\item \textbf{H\"older} \[\mathcal{{H}} = \{ f : [0,1]^n \to [0,1]^k \ \big|\ n,k \in \N,\ \underset{L>0}{\exists}\ \underset{\alpha \in (0,1]}{\exists}\ \|f(x)-f(y)\|_p\leq L\|x-y\|_p^{\alpha} \text{ for all }\ x,y \in [0,1]^n \}\]
		\item \textbf{Linear} \[\mathrm{LIN} = \{ f : [0,1]^n \to [0,1]^k \ \big|\ n,k \in \N,\ \underset{F : \R^n \to \R^k}{\exists}\ F \text{ is linear and } f = F|_{[0,1]^n} \}\]
	\end{enumerate}
\end{defn}

\begin{rem}
Note that we omit $p$ in the notation for classes $\mB$ and $\LIN$, as they clearly do not depend on the choice of the norm. Classes $\mH, \mH_{\alpha}$ and $\mathcal{L}$ are also independent of $p$, as all norms on a finite-dimensional real vector space are equivalent, yet note that the Lipschitz constant of a given map in any of these classes may depend on the choice of the norm.
\end{rem}

\begin{rem}
	Notice that the identity map $\id:[0,1]^{n}\rightarrow[0,1]^{n}$ belongs to all of the
	above classes, except for $\mathcal{L}_L$ and $\mathcal{H}^p_{L, \alpha}$. Indeed for those two classes, we have $\id \in \mathcal{L}_L, \mathcal{H}^p_{L, \alpha}$ if and only if $L\geq 1$.
\end{rem}

Let us define now three compression rates that we will study throughout the paper.

\begin{defn}
	\label{def:ratio is achievable} (See \cite[Definition 3]{WV10}) Let $\mu\in  \mP_{\sigma}([0,1]^\Z)$. Let $\mathcal{{C}},\mathcal{{D}} \subset \{ f : [0,1]^n \to [0,1]^k : n, k \in \N \}$
	be regularity classes. For $n \in \N$ and $\eps \geq 0$, the $\mathcal{C}-\mathcal{D}$ \textbf{almost lossless analog compression rate $\r_{\mathcal{C}-\mathcal{D}}(\mu,\eps,n)\geq0$} of $\mu$ with $n$-block error probability $\eps$ is the infimum of
	\[\frac{k}{n},\]
	where $k$ runs over all natural numbers such that there exist maps
	\[
	f:[0,1]^{n}\rightarrow[0,1]^{k},\ f\in\mathcal{{C}}\ \text{ and }\ g:[0,1]^{k}\rightarrow[0,1]^{n},\ g\in\mathcal{{D}}
	\]
	with
	\[
	\mu(\{x\in [0,1]^\Z |\ g\circ f(x|_{0}^{n-1})\neq x|_{0}^{n-1}\})\leq\eps.
	\]
	Define further:
	\[
	\r_{\mathcal{C}-\mathcal{D}}(\mu,\eps)=\limsup_{n\rightarrow\infty}\ \r_{\mathcal{C}-\mathcal{D}}(\mu,\eps,n).
	\]
\end{defn}

In the next definition, we fix a subshift $\mS \subset [0,1]^\Z$ and consider compressors and decompressors providing almost lossless compression for all measures $\mu \in \mP_{\sigma}(\mS)$. In such a case, compression can be performed without knowledge of the distribution from which the signal comes, as long as the process is supported in $\mS$.

\begin{defn}\label{def:uniform ratio is achievable} Let
	$\mS \subset [0,1]^{\mathbb{Z}}$ be a subshift. Let $\mathcal{{C}},\mathcal{{D}} \subset \{ f : [0,1]^n \to [0,1]^k : n, k \in \N \}$
	be regularity classes. For $n \in \N$ and $\eps \geq 0$, the $\mathcal{C}-\mathcal{D}$ \textbf{uniform almost lossless analog compression rate $\r_{\mathcal{C}-\mathcal{D}}(\mS,\eps,n)\geq0$} of $\mS$ with $n$-block error probability $\eps$ is the infimum of
	\[\frac{k}{n},\]
	where $k$ runs over all natural numbers such that there exist maps
	\[
	f:[0,1]^{n}\rightarrow[0,1]^{k},\ f\in\mathcal{{C}}\ \text{ and }\ g:[0,1]^{k}\rightarrow[0,1]^{n},\ g\in\mathcal{{D}}
	\]
	such that
	\[
	\mu(\{x\in \mS|\ g\circ f(x|_{0}^{n-1})\neq x|_{0}^{n-1}\})\leq\eps \ \text{holds for all } \mu \in \mP_{\sigma}(\mS).
	\]
	Define further:
	\[
	\r_{\mathcal{C}-\mathcal{D}}(\mS,\eps)=\limsup_{n\rightarrow\infty}\ \r_{\mathcal{C}-\mathcal{D}}(\mS,\eps,n).
	\]
\end{defn}

We introduce also an $L^p$ variant.

\begin{defn} \label{def:LP_ratio is achievable}
	Let $\mu\in  \mP_{\sigma}([0,1]^\Z)$. Let $\mathcal{{C}},\mathcal{{D}} \subset \{ f : [0,1]^n \to [0,1]^k : n, k \in \N \}$ 	be regularity classes with respect to the norm $\| \cdot \|$. \footnotemark \footnotetext{recall that this is recorded in the notation, e.g. $\mH^p_{L,\alpha}$ denotes the set of $(L, \alpha)$-H\"{o}lder maps between unit cubes with respect to the norm $\| \cdot \|_p$} For $n \in \N,\ \eps \geq 0$ and $p \in [1, \infty)$, the $\mathcal{C}-\mathcal{D}$ \textbf{$L^p$-analog compression rate $\r_{\mathcal{C}-\mathcal{D}}^{L^p}(\mu,\eps,n)\geq0$} of $\mu$ with $n$-block mean error $\eps$ is the infimum of
	\[\frac{k}{n},\]
	where $k$ runs over all natural numbers such that there exist maps
	\[
	f:[0,1]^{n}\rightarrow[0,1]^{k},\ f\in\mathcal{{C}}\ \text{ and }\ g:[0,1]^{k}\rightarrow[0,1]^{n},\ g\in\mathcal{{D}}
	\]
	with
	\[
	\int \limits_{[0,1]^n} \|x - g \circ f (x)\|^p d (\pi_n)_* \mu (x) \leq \eps^p.
	\]
	Define further:
	\[
	\r_{\mathcal{C}-\mathcal{D}}^{L^p}(\mu,\eps)=\limsup_{n\rightarrow\infty}\ \r_{\mathcal{C}-\mathcal{D}}^{L^p}(\mu,\eps,n).
	\]
\end{defn}

Note that in this work we will consider the $L^p$-compression rate for regularity classes with respect to the matching norm $\|\cdot\|_p$. 

\begin{rem}\label{r:rates_infty}
All of the above compression rates are non-increasing in $\eps$, i.e. if $\eps \leq \eps'$ then $r(\mu,\eps,n)\geq r(\mu,\eps',n)$ and $r(\mu,\eps)\geq r(\mu,\eps')$. Let $\mathcal{{C}},\mathcal{{D}}$ be regularity classes such that $\id|_{[0,1]^n}\in\mathcal{{C}},\mathcal{{D}}$ for every $n \in \N$. Then $r(\mu,\eps) \leq 1$ holds for all $\mu\in\mathcal{{P}}_{\sigma}(\mS)$ and $\eps \geq 0$.
	If $\mC$ and $\mD$ are regularity classes with respect to a norm such that $\diam([0,1]^n, \|\cdot\|)=1$ for every $n \in \N$ (note that we normalize norms $\| \cdot \|_p,\ p \in [1,\infty]$ such that this holds), then
	\begin{equation}\label{e:rate_comparision}
 \r_{\mathcal{C}-\mathcal{D}}^{L^p}(\mu,\eps) \leq \r_{\mathcal{C}-\mathcal{D}}(\mu,\eps^p).
	\end{equation}
	for $1 \leq p < \infty$ and $\mu \in \mP_{\sigma}(\mS)$.\\
\end{rem}

\begin{rem}\label{rem:extension_theorems}
	One could have used a seemingly more permissive definition with the compressor function defined only on a set of words of length $n$ occuring in $\mS$, i.e. with conditions \begin{equation}\label{eq:restriced_compressors}
	f:\pi_{n}(\mS)\subset[0,1]^{n}\rightarrow[0,1]^{k},f\in\mathcal{{C}} \text{ and } g:f(\pi_{n}(\mS))\subset[0,1]^{k}\rightarrow[0,1]^{n},g\in\mathcal{{D}}.
	\end{equation}
	However, if one can apply suitable extension theorems to extend functions $f,g$ to the full cubes $[0,1]^n$ and $[0,1]^k$, respectively, while maintaining assumed regularity, then this would have resulted with an equivalent definition. All of the classes in Definition \ref{def:(regularity-classes)} other than $\mathrm{LIN}$ admit such an extension when considered with respect to the $\|\cdot\|_{\infty}$ norm. In particular, for classes $\mathcal{H}_{L, \alpha}^\infty$ this follows from Banach's theorem on extension of H\"older and Lipschitz functions, stating that a real-valued function defined on a subset $A$ of a metric space and satisfying $|f(x)-f(y)|\leq Ld(x,y)^{\alpha}$ for all $x,y\in A$, can be extended to the whole metric space so as to satisfy the same inequality (\cite[Theorem IV.7.5]{Banach51}, see also \cite{minty1970extension}). Therefore, given a function $f : A \to [0,1]^k,\ A \subset [0,1]^n,\ f = (f_1, \ldots, f_k)$ satisfying $\|f(x) - f(y)\|_{\infty} \leq L \|x-y\|_{\infty}^{\alpha}$ on $A$, one can apply Banach's theorem to coordinate functions $f_1, \ldots, f_k$ (clearly satisfying $|f_j(x) - f_j(y)| \leq L \|x - y\|^{\alpha}_{\infty}$), obtaining extensions $\tilde{f_1}, \ldots, \tilde{f_k} : [0,1]^n \to \R$. Then $\overline{f}_1, \ldots , \overline{f}_k : [0,1]^n \to [0,1]$ given as $\overline{f}_j(x) = \max \{ \min\{ f(x), 1 \}, 0 \}$ are again $(L,\alpha)$-H\"{o}lder and satisfy $\overline{f}_j = f_j$ on $A$. Therefore $\overline{f} = (\overline{f}_1, \ldots, \overline{f}_k)$ is an extension of $f$ from $A$ to $[0,1]^n$ and
	\[ \|\overline{f}(x) - \overline{f}(x)\|_{\infty} = \max \limits_{1 \leq j \leq k} \|\overline{f}_j(x) - \overline{f}_j(y)\| \leq L \|x - y\|^{\alpha}_{\infty}.\]
	This yields an extension theorem for $\mathcal{H}_{L, \alpha}^\infty$. By \cite[Theorem 1]{minty1970extension}, analogous result is true for the class $\mathcal{H}_{L, \alpha}^2$. The above considerations apply also to $\mathcal{H}_{L, \alpha}^p$ for other $p \in [1, \infty)$, however the extended function might have a worse Lipschitz constant $L$. Namely, repeating the above argument for a function $f : A \to [0,1]^k$ satisfying $\|f(x) - f(y)\|_p \leq L \|x-y\|_p^{\alpha}$, the coordinate functions $f_j$ will satisfy $|f_j(x) - f_j(y)| \leq k^{\frac{1}{p}} \|f(x) - f(y)\|_p \leq k^{\frac{1}{p}}L \|x-y\|_p^{\alpha}$. Therefore, the extended function $\overline{f}$ is guaranteed to satisfy
	\[ \|f(x) - f(y)\|_p \leq k^{\frac{1}{p}}L \|x - y\|_p^{\alpha}. \]
In other words, $(L,\alpha)$-H\"{o}lder function in the norm $\|\cdot\|_p$ can be extended to an element of $\mH^p_{k^\frac{1}{p}L,\alpha}$. It is known that in general one cannot have an extension theorem for H\"{o}lder maps preserving the Lipschitz constant for $p \in (1,\infty) \setminus \{ 2\}$ (see \cite[Section 1.30]{SchwartzNFA}). In conclusion, if $\mC, \mD \in \{ \mB, \LIN, \mH_{\alpha}, \mH^{\infty}_{L,\alpha}, \mH^{2}_{L,\alpha}\}$, then (\ref{eq:restriced_compressors}) will result in equivalent definitions of compression rates.
\end{rem}

\section{Mean dimensions}\label{sec:dim}

In this section we will introduce two dynamical notions of mean dimensions: \emph{metric mean dimension} and \emph{mean box dimension}. They attempt to capture the average number of dimensions per time unit required to describe a given system and can serve as complexity measures of a subshift. The main results of this paper are certain bounds on compression rates in terms of these mean dimensions.

\subsection{Box dimension}
Let us begin by introducing the non-dynamical notion of box-counting dimension. It will serve a basis for dynamical definitions of mean dimensions.
\begin{defn}\label{def:hashtag} Let $(\mX, \rho)$ be a compact metric space. For $\eps > 0$, the \textbf{$\eps$-covering number} of a subset $A \subset \mX$, denoted by $\#(A,\rho,\varepsilon)$, is the minimal cardinality $N$ of an open cover $\{U_{1},\dots,U_{N}\}$ of $A$ such that all $U_{n}$ have diameter smaller than $\varepsilon$. The \textbf{$\eps$-net} in $A$ is a subset $E \subset A$ such that for every $x \in A$ the distance $\dist(x, E) := \inf \limits_{y \in E} \rho(x,y)$ is strictly smaller than $\eps$.
\end{defn}
Note that in $A$ there exists an $\eps$-net of cardinality $\#(A, \rho, \eps)$. When the metric $\rho$ is coming from a norm $\| \cdot \|$ (i.e. $\rho(x,y) = \| x - y \|$), we will write simply $\#(A,\| \cdot \|,\eps)$ for $\#(A,\rho,\varepsilon)$.

\begin{defn}
	Let $(\mX, \rho)$ be a compact metric space. The \textbf{upper box-counting (Minkowski) dimension} of $A\subset\mX$
	is defined as
	\[\udim(A)=\limsup \limits_{\varepsilon\to0}\frac{\log\#(A,\rho,\varepsilon)}{\log\frac{1}{\varepsilon}}.\]
	
	Similarly the \textbf{lower box-counting dimension} of $A$ is defined as
	\[\ldim(A)=\liminf \limits_{\varepsilon\to0}\frac{\log\#(A,\rho,\varepsilon)}{\log\frac{1}{\varepsilon}}.\]
	If upper and lower limits coincide, then we call its common value the \textbf{box-counting dimension} of $A$ and denote it by $\mathrm{dim}_B(A)$.
\end{defn}

In the sequel we will consider primarily sets $A \subset [0,1]^n$ with the distance induced by the norm $\| \cdot \|_{\infty}$. Note that, any other equivalent metric on $[0,1]^n$ (e.g. $\| \cdot \|_p,\ p \in [1, \infty)$) will give the same values of box-counting dimensions. Some of the basic properties of box-counting dimensions are summarized in the next proposition (see e.g. \cite[Prop. 3.3 and 3.4]{Rob11}).
\begin{prop}\label{prop:dim_properties}
Let $A,B$ be subsets of compact metric spaces. Then
\begin{enumerate}
\item if $A \subset B$, then $\udim(A) \leq \udim(B)$ and $\ldim(A) \leq \ldim(B)$
\item\label{bdim_sum} $\udim(A \cup B) \leq \max\{ \udim(A), \udim(B)\}$
\item\label{bdim_submult} $\udim(A \times B) \leq \udim(A) + \udim(B)$
\item $\dim_B([0,1]^n)=n$
\end{enumerate}
\end{prop}

 For more on the dimension theory see \cite{falconer2004fractal}, \cite{mattila} and \cite{Rob11}.
\subsection{Metric mean dimension and mean box dimension}
The notion that we will be mostly interested in is the \emph{metric mean dimension}:

\begin{defn}\label{df:mmdim_subshift} Let $\mS \subset [0,1]^\Z$ be a subshift. 	The upper and lower \textbf{metric mean dimensions} of the system $\mS$ are defined as
	\[
	\ummdim(\mS)=\limsup_{\varepsilon\to0}\lim_{n\rightarrow\infty}\frac{\log\#(\pi_{n}(\mS),||\cdot||_{\infty},\varepsilon)}{n\log\frac{1}{\varepsilon}}
	\]
and	
	\[
	\lmmdim(\mS)=\liminf_{\varepsilon\to0}\lim_{n\rightarrow\infty}\frac{\log\#(\pi_{n}(\mS),||\cdot||_{\infty},\varepsilon)}{n\log\frac{1}{\varepsilon}}.
	\]
	(the limit with respect to $n$ exists as the sequence $n \mapsto \log\#(\pi_{n}(\mS),||\cdot||_{\infty},\varepsilon)$ is subadditive). If upper and lower limits coincide, then we call its common value the \textbf{metric mean dimension} of $\mS$ and denote it by $\mmdim(\mS)$.
\end{defn}

The study of mean dimension type invariants in dynamical systems goes back to Gromov \cite{G}. The notion of metric mean dimension was introduced by Lindenstrauss and Weiss in \cite{LW00}. The above formula is valid for subshifts in $[0,1]^\Z$, however metric mean dimension was originally  defined by a different (but equivalent) formula, which applies to any dynamical system on a compact metric space (see Proposition \ref{prop:canonical}). In Appendix \ref{app:mmdim} we present the general definition, the proof of equivalence and more remarks on metric mean dimension from the point of view of dynamical systems. We will consider one more notion of mean dimension. It is a box-counting dimension version of the projective dimension introduced by Gromov in \cite{G}.

\begin{defn}
For a subshift $\mS\subset[0,1]{}^{\mathbb{{Z}}}$ we define its upper \textbf{mean
	box dimension} as
\[
\mbdim(\mS)=\lim_{n\to\infty}\frac{\overline{\dim}_{B}(\pi_{n}(\mS))}{n},
\]
where $\udim(\pi_n(\mS))$ is calculated with respect to any $\| \cdot \|_p$ norm on $[0,1]^n$ (the limit exists as the sequence $n\mapsto\overline{\dim}_{B}(\pi_{n}(\mS))$ is subadditive, since $\pi_{n+k}(\mS) \subset \pi_n(\mS) \times \pi_k(\mS)$).
\end{defn}

Note that $\mbdim(\mS)$ is defined as $\ummdim(\mS)$, but with exchanged order of limits. They fulfill the following inequality, which is an analog of Gromov's pro-mean inequality \cite[Subsection 1.9.1]{G}.
\begin{prop}
	\label{prop:mdim leq mbdim} Let $\mS\subset[0,1]^{\Z}$ be a subshift. Then
	\[
	\overline{\mdim}_{M}(\mS)\leq\mbdim(\mS).
	\]
\end{prop}

For the proof see Appendix \ref{app:proof_pro_mean_ineq}.
\subsection{Examples}
Let us present now several examples (with some of the calculations postponed to Appendix \ref{app:examples}).
\begin{example}\label{ex:general_full_shift} Let $A \subset [0,1]$ be a closed subset. Then $\mS = A^\Z \subset [0,1]^\Z$ is a subshift satisfying
\begin{equation}\label{eq:full_shift_mmdim_mbdim}
\ummdim(\mS) = \mbdim(\mS) = \udim(A).
\end{equation}
By Proposition \ref{prop:mdim leq mbdim} it is enough to prove
\[ \udim(A) \leq  \ummdim(\mS)\ \text{ and }\ \mbdim(\mS) \leq \udim(A) \] 
The lower bound follows from \cite[Thm. 5]{velozo2017rate}. For the upper bound note that $\pi_n(\mS) = A^n$, hence by Proposition \ref{prop:dim_properties}.\ref{bdim_submult}, $\udim(\pi_n(\mS)) = \udim(A^n) \leq n \udim(A)$.
\end{example}

In general, equality does not hold in Proposition \ref{prop:mdim leq mbdim}:
\begin{example}\label{e:mmdim_0_mbdim_1}
	For $m\geq1$ let
	\[A_{m}:=\{ (\ldots, 0, 0)\} \times [0,\frac{1}{2^{m}}]^{m} \times \{ (0,0, \ldots) \}\subset[0,1]^{\Z},\]
	where the cube $[0,\frac{1}{2^{m}}]^{m}$ is located on coordinates $0, 1, \ldots, m-1$. Set
	\[ \mS_m = \bigcup\limits_{n\in \Z}\sigma^{n}(A_{m}) \text{ and }\mS:=\bigcup\limits_{m\geq1} \mS_m.\]
	Then $\mS$ is a subshift satisfying
	\[ \mmdim(\mS) = 0\ \text{ and }\ \mbdim(\mS) = 1.\]
	For calculations see Appendix \ref{app:mmdim_0_mbdim_1}.
\end{example}

Let us now consider an example which arises naturally in the context of compressed sensing for sparse signals. It will be our main example throughout the paper.

\begin{example}\label{ex:sparse_subshift}
Fix $K, N \in \N$ with $K \leq N$. Let $\mS \subset [0,1]^\Z$ be the subshift consisting of $(N, K)$-\emph{sparse} vectors, i.e.
\[ \mS = \{ x \in [0,1]^\Z : \underset{j \in \Z}{\forall}\ \|x|_{j}^{j+N-1} \|_0 \leq K\}. \]
We will call $\mS$ the \textbf{$(N,K)$-sparse subshift}. It satisfies
\[ \mmdim(\mS) = \mbdim(\mS) = \frac{K}{N}. \]
Let us prove these equalities. By Proposition \ref{prop:mdim leq mbdim}, it is enough to prove
\[\frac{K}{N} \leq \lmmdim(\mS)\ \text{ and }\ \mbdim(\mS) \leq \frac{K}{N}.\] For the upper bound, observe that $\pi_{lN}(\mS)$ for $\ell \in \N$ is a finite union of sets of the upper box-counting dimension at most $\ell K$ (as every $x \in \pi_{\ell N}(\mS)$ satisfies $\|x\|_0 \leq \ell K$). Consequently, by Proposition \ref{prop:dim_properties}.\ref{bdim_sum}, $\udim(\pi_{\ell N}(\mS)) \leq \ell K$, hence $\mbdim(\mS) \leq \frac{K}{N}$. For the lower bound, note that $\pi_{\ell N}(\mS)$ contains the set $A_\ell := ([0,1]^K \times \{ 0 \}^{N-K})^\ell$. For fixed $\eps>0$ we have therefore (see e.g. \cite[Section 2.2]{falconer2004fractal})
\[ \#(\pi_{lN}(\mS),||\cdot||_{\infty},\eps) \geq \#(A_\ell,||\cdot||_{\infty},\eps) \geq \Big(\frac{1}{\eps}\Big)^{K\ell}, \]
hence
\[ \lmmdim(\mS) = \liminf \limits_{\eps \to 0}\ \lim_{\ell \rightarrow\infty}\frac{\log\#(\pi_{\ell N}(\mS),||\cdot||_{\infty},\varepsilon)}{\ell N\log\frac{1}{\varepsilon}} \geq \liminf \limits_{\eps \to 0}\ \lim_{\ell \rightarrow\infty}\ \frac{\ell K \log \frac{1}{\eps} }{\ell N \log \frac{1}{\eps}} = \frac{K}{N}. \]
\end{example}

\section{Main results}\label{sec:main}
\numberwithin{thm}{section}
In this section we present the main results of the paper. Instead of assuming specific properties for the measure governing the source, we consider the setting in which the set of all possibles trajectories is known and we look for compression rates corresponding to the worst-case scenario. Therefore the following question is our main question:\\ \ \\
\textbf{Main Question:} Given a subshift $\mS \subset [0,1]^\Z$, calculate
\[\sup \limits_{\eps > 0} \sup \limits_{\mu\in\mathcal{P}_{\sigma}(\mS)}\r_{\mathcal{{C}}-\mathcal{{D}}}(\mu,\eps), \ \sup \limits_{\eps > 0} \ \sup \limits_{\mu\in\mathcal{P}_{\sigma}(\mS)} \ \r^{L^p}_{\mathcal{{C}}-\mathcal{{D}}}(\mu,\eps),\ \sup \limits_{\eps > 0}\ \r_{\mathcal{{C}}-\mathcal{{D}}}(\mS,\eps) \text{ and } \r_{\mathcal{C}-\mathcal{D}}(\mS, 0)\] for fixed regularity classes $\mathcal{C}$ and $\mathcal{D}$.\\

The above suprema correspond to rates which can be achieved whenever geometric information on the signal is given (i.e. $\mS$ is known), but its statistics are unknown (i.e. the particular $\mu \in \mS$ which governs the source is not known) and can be a priori given by any stationary distribution supported in $\mS$.

We will be mainly interested in this question for $\mathcal{C} \in \{\mathcal{B}, \mathrm{LIN}\}$ and $\mathcal{D}=\mathcal{H}^p_{L,\alpha}$. Such or similar regularity conditions have appeared previously in the literature (e.g. in \cite{WV10, jalali2017universal, CompressionCS17}. See Section \ref{sec:related_works} for a more detailed discussion). Our goal is to connect the above uniform compression rates with geometric-dynamical invariants of the given subshift $\mS$ - its mean dimensions.

\begin{rem} Note that for any compression variant it holds
	\[ \sup \limits_{\mu\in\mathcal{P}_{\sigma}(\mS)} \ \sup \limits_{\eps > 0} \  \r_{\mathcal{{C}}-\mathcal{{D}}}(\mu,\eps) = \sup \limits_{\eps > 0} \ \sup \limits_{\mu\in\mathcal{P}_{\sigma}(\mS)} \ \r_{\mathcal{{C}}-\mathcal{{D}}}(\mu,\eps) = \]
	\[ = \lim \limits_{\eps \to 0} \  \sup \limits_{\mu\in\mathcal{P}_{\sigma}(\mS)} \ \r_{\mathcal{{C}}-\mathcal{{D}}}(\mu,\eps) = \sup \limits_{\mu\in\mathcal{P}_{\sigma}(\mS)} \ \lim \limits_{\eps \to 0} \   \r_{\mathcal{{C}}-\mathcal{{D}}}(\mu,\eps). \]
\end{rem}
Below we formulate the main results of the paper. For their proofs see Section \ref{sec:proofs}.

\subsection{Lower bounds}

\numberwithin{thm}{subsection}

Our main result is a worst-case fundamental limit on $\sup \limits_{\eps>0}\ \r^{L^p}_{\mathcal{B} - \mathcal{H}^p_{L, \alpha}}(\mu, \eps)$ (and therefore on $\sup \limits_{\eps>0}\ \r_{\mathcal{B} - \mathcal{H}^p_{L, \alpha}}(\mu, \eps)$ -- see (\ref{e:rate_comparision})) for $\mu \in \mP_{\sigma}(\mS)$ in terms of the metric mean dimension of $\mS$. This corresponds to the setting where the distribution of the signal is not known, but the set $\mS$ of all possible trajectories is known (i.e. one has knowledge of the words that can occur in the input, but not their probabilities) and one is interested in the worst (largest) compression rate which is required to code the signal.

\begin{thm}\label{thm:main_holder_low} Let $\mS \subset [0,1]^\Z$ be a subshift. The following holds for
	every $0<\alpha \leq 1, L>0$ and $p \in [1, \infty)$:
	\[
	\alpha \ummdim(\mS)\leq \sup \limits_{\eps>0}\ \sup \limits_{\mu\in  \mP_{\sigma}(\mS)}\ \r_{\mathcal{{B}}-\mathcal{{H}}_{L,\alpha}^p}^{L^p}(\mu,\eps) \leq \sup \limits_{\eps>0}\ \sup \limits_{\mu\in  \mP_{\sigma}(\mS)}\ \r_{\mathcal{{B}}-\mathcal{{H}}_{L,\alpha}^p}(\mu,\eps).\]
	Similarly for $p=\infty$
	\[ \alpha \ummdim(\mS)\leq \sup \limits_{\eps>0}\ \sup \limits_{\mu\in  \mP_{\sigma}(\mS)}\ \r_{\mathcal{{B}}-\mathcal{{H}}_{L,\alpha}^\infty}(\mu,\eps). \]
\end{thm}

Intuitively, this result states that for a given set of trajectories $\mS$, we can always find a stationary stochastic process supported in $\mS$, which cannot be compressed at a better rate than $\alpha \ummdim(\mS)$. Note that we work with Borel compressors, which allows very general compression schemes (e.g. non-continuous quantizations) for which this bound can be applied. We consider H\"{o}lder decompressors with both H\"{o}lder exponent $\alpha$ and Lipschitz constant $L$ fixed, i.e. decompressors have to fulfill $(L,\alpha)$-H\"{o}lder condition for \emph{every} blocklength (hence uniform control of the decompression error is guaranteed for arbitrary long messages). We consider compressors and decompressors as functions between finite-dimensional unit cubes, hence the Lipschitz constant of the decompressor cannot be arbitrary reduced as this necessitates expanding the image of the compressor to beyond the unit cube - this makes the above constraint non-trivial. As shown by examples in the next section, fixing $L$ (i.e. choosing $\mD = \mH^p_{L,\alpha}$ instead of $\mD = \mH_{\alpha}$) is necessary for the above bound to hold (see Example \ref{ex:L_must_stay}). It is not sharp in general (see Example \ref{ex:two_symbols_shift}), yet equality holds for some subshifts after letting $L \to \infty$ (we prove it for $(N,K)$-sparse subshifts in Example \ref{ex:sparse_b-h_equality}). Such equality, whenever holds, can be seen as an operational characterization of metric mean dimension. For the proof of Theorem \ref{thm:main_holder_low} see Section \ref{sec:lower}.

We obtain also a lower bound on $\r_{\mathcal{{B}}-\mathcal{{H}}^2_{L,\alpha}}^{L^2}(\mu, \eps)$ for fixed $\mu \in \mP_{\sigma}(\mS)$ and $\eps>0$ (see Theorem \ref{thm:main}) and on $\sup \limits_{\eps > 0}\  \r_{\mathcal{{B}}-\mathcal{{H}}_{L,\alpha}}^{L^2}(\mu, \eps)$  for fixed $\mu \in \mP_{\sigma}(\mS)$ in terms of the rate-distortion dimension of a measure (see Corollary \ref{cor:rate_lower_fixed_measure} for details).

\subsection{Upper bounds}

The lower bound of Theorem \ref{thm:main_holder_low} is accompanied by two upper bounds. The first one, concerning linear compressors and $(L,\alpha)$-H\"{o}lder decompressors, gives an upper bound on the worst-case compression rate for zero-probability error in terms of mean box dimension. It is an application of the finite-dimensional linear embedding theorem with H\"{o}lder inverse.

\begin{thm}\label{thm:main_holder_up}
Let $\mS \subset [0,1]^\Z$ be a subshift. Then, for
	every $0<\alpha < 1$ and $p \in [1, \infty]$
	\[ \inf \limits_{L>0}\ \r_{\mathrm{LIN}-\mathcal{{H}}_{L,\alpha}^p}(\mS,0) \leq  \frac{2}{1-\alpha}\mbdim(\mS). \]
\end{thm}
By the definition of $\r_{\mathrm{LIN}-\mathcal{{H}}_{L,\alpha}^p}(\mS,0)$, the sequence of coders and decoders which achieves the above upper bound is independent of $\mu \in \mP_{\sigma}(\mS)$, hence compression for a given distribution $\mu \in \mP_{\sigma}(\mS)$ can be performed without prior knowledge of $\mu$, as long as $\mu$ is supported in $\mS$.
The infimum over $L$ means that we can obtain compression rates arbitrary close to $\frac{2}{1-\alpha}\mbdim(\mS)$ from above, at the cost of increasing $L$ (which, once fixed, is valid for all blocklengths). Since $\r_{\mathrm{LIN}-\mathcal{{H}}^p_{L,\alpha}}(\mS,0) \leq 1$ for $L \geq 1$ (as identity belongs to both classes $\mathrm{LIN}$ and $\mathcal{H}^p_{L,\alpha}$), we always have $\min \{1, \frac{2}{1-\alpha}\mbdim(\mS)\}$ as the upper bound. For that reason, Theorem \ref{thm:main_holder_up} gives a meaningful bound as long as $\mbdim(\mS) < \frac{1-\alpha}{2}$. Combining Theorems \ref{thm:main_holder_low} and \ref{thm:main_holder_up} we obtain the following corollary:

\begin{cor}\label{cor:main_bounds_joined}
Let $\mS \subset [0,1]^\Z$ be a subshift. For every $0<\alpha < 1$ and $p \in [1, \infty)$ the following holds:
	\[ \alpha \ummdim(\mS) \leq \inf \limits_{L>0}\ \sup \limits_{\eps>0}\ \sup \limits_{\mu\in  \mP_{\sigma}(\mS)} \r_{\mathcal{{B}}-\mathcal{{H}}_{L,\alpha}^p}^{L^p}(\mu,\eps)  \leq \inf \limits_{L>0}\ \sup \limits_{\eps>0}\ \sup \limits_{\mu\in  \mP_{\sigma}(\mS)} \r_{\mathcal{{B}}-\mathcal{{H}}_{L,\alpha}^p}(\mu,\eps) \leq \]
	\[ \leq \inf \limits_{L>0}\ \sup \limits_{\eps>0}\ \sup \limits_{\mu\in  \mP_{\sigma}(\mS)} \r_{\mathrm{LIN}-\mathcal{{H}}_{L,\alpha}^p}(\mu,\eps) \leq \min \{1, \frac{2}{1-\alpha}\mbdim(\mS)\}. \]
	Similarly for $p=\infty$
	\[ \alpha \ummdim(\mS)  \leq \inf \limits_{L>0}\ \sup \limits_{\eps>0}\ \sup \limits_{\mu\in  \mP_{\sigma}(\mS)} \r_{\mathcal{{B}}-\mathcal{{H}}_{L,\alpha}^\infty}(\mu,\eps) \leq \]
\[ \leq \inf \limits_{L>0}\ \sup \limits_{\eps>0}\ \sup \limits_{\mu\in  \mP_{\sigma}(\mS)} \r_{\mathrm{LIN}-\mathcal{{H}}_{L,\alpha}^\infty}(\mu,\eps) \leq \min \{1, \frac{2}{1-\alpha}\mbdim(\mS)\}. \]
\end{cor}

For the proof of Theorem \ref{thm:main_holder_up} see Section \ref{sec:upper}.

The second upper bound concerns the setting of Theorem \ref{thm:main_holder_low}: Borel compression and $(L,\alpha)$-H\"{o}lder decompression. It is a universal upper bound, independent of $\mS$. Moreover, the compressor and decompressor functions are independent of the choice of measure $\mu \in \mP_{\sigma}(\mS)$ and operate with zero-probability error.

\begin{prop}\label{prop:peano_up}
Let $\mS \subset [0,1]^\Z$ be a subshift. Then for every $0 < \alpha \leq 1$ and $p \in [1, \infty]$
	\[ \inf \limits_{L>0}\ \r_{\mathcal{B} - \mathcal{H}_{L, \alpha}^{p}} (\mS, 0) \leq \alpha. \]
\end{prop}
Note that this upper bound is general is a sense that it does not depend on $\mS$ (in particular, it is true for the full shift $\mS = [0,1]^\Z$). The proof is based on the existence of H\"{o}lder surjective maps between unit cubes (constructions akin to the Peano curve construction). See Section \ref{sec:peano}.

Proposition \ref{prop:peano_up} together with Theorem \ref{thm:main_holder_low} gives the following inequalities.

\begin{cor}\label{cor:rates_full_mmdim}
Let $\mS \subset [0,1]^\Z$ be a subshift. For every $0<\alpha \leq 1$ and $p \in [1, \infty)$ the following holds:
	\[ \alpha \ummdim(\mS) \leq \inf \limits_{L>0}\ \sup \limits_{\eps>0}\ \sup \limits_{\mu\in  \mP_{\sigma}(\mS)} \r_{\mathcal{{B}}-\mathcal{{H}}_{L,\alpha}^p}^{L^p}(\mu,\eps)  \leq \inf \limits_{L>0}\ \r_{\mathcal{B} - \mathcal{H}_{L, \alpha}^{p}} (\mS, 0) \leq \alpha. \]
\end{cor}
In particular, if $\ummdim(\mS)=1$, then the above compression rates are all equal to $\alpha$. 
\section{Examples}\label{sec:examples}

In this section we present several examples illustrating results of the previous section. For the proofs and calculations see Appendix \ref{app:examples}.

\subsection{Discrete signals}

In general, equality does not hold in Theorem \ref{thm:main_holder_low}.

\begin{example}\label{ex:two_symbols_shift}
	Let $\mS := \{0,1\}^{\Z}$. Then $\mmdim(\mS)= \dim_B(\{0,1\}) = 0$ (see Example \ref{ex:general_full_shift}). For any $ \alpha \in (0,1]$ and $L>0$ it holds that
\begin{equation}\label{eq:two_symbols_shift_lower} \sup \limits_{\varepsilon > 0} \ \sup \limits_{\mu \in  \mP_{\sigma}(\mS)} \r^{L^1}_{\mathcal{B} - \mathcal{H}_{L, \alpha}^1}(\mu, \eps) \geq \frac{\alpha (1 - H(\frac{1}{4}))}{\log \big( \max \{ 8 L, 8 \} \big)} > 0,
\end{equation}
where $H(p) = -p\log(p) - (1-p)\log(1-p)$ is the Shannon entropy of the probability vector $(p,1-p)$. For the proof of (\ref{eq:two_symbols_shift_lower}) see Appendix \ref{sub:two_symbols_shift}.
\end{example}
This example shows that equality in Theorem  \ref{thm:main_holder_low} might not hold for any fixed $L>0$. Note however that the above lower bound tends to $0$ as $L \to \infty$.

The next example shows that one cannot change the class $\mathcal{H}^p_{L, \alpha}$ to $\mathcal{H}_{\alpha}$ in Theorem \ref{thm:main_holder_low}.
\begin{example}\label{ex:L_must_stay}
	Let $A = \{0\} \cup \{\frac{1}{n} : n \in \N \}$ and $\mS := A^{\Z}$. Then, according to Example \ref{ex:general_full_shift} and \cite[Example 2.7]{falconer2004fractal}
	\[\ummdim(\mS) = \mbdim(\mS) = \udim(A) = \frac{1}{2}.\]
	On the other hand, we will prove
	\[\sup \limits_{\eps > 0}\ \sup \limits_{\mu \in  \mP_{\sigma}(\mS)}\ \r_{\mathcal{B} - \mathcal{H}_{\alpha}}(\mu, \eps) = 0 \text{ for every }\alpha \in (0,1].\]
Actually, a stronger claim is true:
\[\sup \limits_{\eps > 0}\  \sup \limits_{\mu \in  \mP_{\sigma}(\mS)}\ \r_{\mathrm{LIN} - \mathcal{L}}(\mu, \eps) = 0.\]
For the proof of above statements see Appendix \ref{sub:L_must_stay}. 
\end{example}

\subsection{Sparse signals}
It turns out that the lower bound in Theorem \ref{thm:main_holder_low} is tight (after taking infimum over Lipschitz constants, i.e. letting $L \to \infty$) in the case of subshifts consisting of $(N,K)$-sparse vectors.

\begin{example}\label{ex:sparse_b-h_equality}
Let $\mS \subset[0,1]^\Z$ be the subshift from Example \ref{ex:sparse_subshift}, i.e.
\[ \mS = \{ x \in [0,1]^\Z : \underset{j \in \Z}{\forall}\ \|x|_{j}^{j+N-1} \|_0 \leq K\} \]
for $N,K \in \N,\ K \leq N$. Recall that $\mmdim(\mS) = \mbdim(\mS) = \frac{K}{N}$. For every $p \in [1,\infty]$ and $\alpha \in (0,1]$ it holds
\begin{equation}\label{eq:sparse_borel_rate} \inf \limits_{L>0}\ \sup \limits_{\mu \in \mP_{\sigma}(\mS)} \sup \limits_{\eps>0}\ \r_{\mB - \mH^p_{L,\alpha}}(\mu, \eps) = \inf \limits_{L>0}\ \r_{\mB - \mH^p_{L,\alpha}}(\mS, 0) = \alpha \mmdim(\mS) = \frac{\alpha K}{N}.
\end{equation}

For the proof see Appendix \ref{sec:sparse_b-h_equality}.
\end{example}

Such strong compression properties are, similarly as in Proposition \ref{prop:peano_up}, obtained by employing constructions akin to the Peano curve construction between unit cubes.

The following example gives a lower bound on compression rates for the linear compression with H\"{o}lder decompression in case of the $(N,K)$-sparse subshift. It does not match the upper bound from Theorem \ref{thm:main_holder_up}, but it shows that one cannot improve (in general) the constant $\frac{2}{1 - \alpha}$ to $\frac{t}{1 - \alpha}$ for any $t < 2$.

\begin{example}\label{ex:sparse_LIN-H}
Let $\mS \subset[0,1]^\Z$ be the subshift from Example \ref{ex:sparse_subshift}, i.e.
\[ \mS = \{ x \in [0,1]^\Z : \underset{j \in \Z}{\forall}\ \|x|_{j}^{j+N-1} \|_0 \leq K\} \]
for $K,N \in \N,\ K \leq N$. For any $\alpha \in (0,1]$ inequality
\begin{equation}\label{eq:LIN-H_sparse_lower} \sup \limits_{\mu \in \mP_{\sigma}(\mS)} \r_{\LIN - \mH_{\alpha}}(\mu, 0) \geq \min \Big\{\frac{2K}{N}, 1 \Big\} = \min \{ 2\mmdim(\mS), 1 \}
\end{equation}
holds. Consequently, for every $p \in [1,\infty]$
\[ \min \{ 2\mmdim(\mS), 1 \} \leq  \inf \limits_{L>0}\ \r_{\LIN - \mH^p_{L, \alpha}}(\mS, 0) \leq \min \{ \frac{2}{1-\alpha}\mmdim(\mS), 1 \}.  \]
For the proof see Appendix \ref{sub:sparse_LIN-H}. Let us emphasize that the proof of (\ref{eq:LIN-H_sparse_lower}) does not use any other properties of $\alpha$-H\"{o}lder maps than continuity. Therefore the lower bound (\ref{eq:LIN-H_sparse_lower}) holds also for merely continuous decompressors.
\end{example}

\section{Variational principle for metric mean dimension}\label{sec:rate_dist}
In this section we present the main tool of the proof of Theorem \ref{thm:main_holder_low}: a variational principle for metric mean dimension, which expresses the metric mean dimension of a subshift $\mS$ in terms of rate-distortion functions corresponding to the stationary stochastic processes supported in $\mS$. It was proved by Lindenstrauss and Tsukamoto in \cite{lindenstrauss_tsukamoto2017rate}. We present its formulation which is suitable for our applications. It differs slightly from the one in \cite{lindenstrauss_tsukamoto2017rate}. In Appendix \ref{app:var_prin} we deduce the formulation below from the original result of  \cite{lindenstrauss_tsukamoto2017rate}. 
\subsection{Rate-distortion function}
The information-theoretic notion used in the formulation of the variational principle is the \emph{rate-distortion function}. We provide a slight modification of the expression used in \cite{lindenstrauss_tsukamoto2017rate}, better suited for the setting of subshifts (and closer to the standard definition, see e.g. \cite{Gray11}). We state the definition for subshifts with alphabet space being a general compact metric space $(A,d)$, however in this work we will consider mainly the case $A = [0,1]$ with the standard metric.

\begin{defn}\label{def:rate_distortion_function}
Let $(A,d)$ be a compact metric space. Let
	$\mS \subset A^{\mathbb{Z}}$ be a subshift and $\mu\in  \mP_{\sigma}(\mS)$. For $p \in [1, \infty),\ \varepsilon>0$ and $n\in\mathbb{{N}}$
	we define the \textbf{$L^p$ rate-distortion function} $R_{\mu,p}(n,\varepsilon)$
	as the infimum of
	\begin{equation}
	\frac{I(X;Y)}{n},\label{eq: definition of rate-distortion function}
	\end{equation}
	where $X=(X_0, \ldots, X_{n-1})$ and $Y=(Y_{0},\dots,Y_{n-1})$ are random variables defined on some probability space $(\Omega,\mathbb{P})$ such that

	\begin{itemize}
		\item $X=(X_0, \ldots, X_{n-1})$ takes values in $A^n$, and
		its law is given by $(\pi_n)_*\mu$.
		\item $Y=(Y_{0},\dots,Y_{n-1})$ takes values in $A^n$ and approximates $(X_{0},X_{1},\dots,X_{n-1})$
		in the sense that
		\begin{equation}\label{eq: distortion condition}
		\begin{gathered}
		\mathbb{{E}}\left(\frac{1}{n}\sum_{k=0}^{n-1}d(X_{k},Y_{k})^p\right)\leq\varepsilon^p.
		\end{gathered}
		\end{equation}
	\end{itemize}
	Here $\mathbb{E}(\cdot)$ is the expectation with respect to the probability
	measure $\mathbb{P}$ and $I(X;Y)$ is the mutual information of random vectors $X$ and $Y$ (see \cite{Gray11} and \cite{lindenstrauss_tsukamoto2017rate}).
	As proved in \cite[Theorem 9.6.1]{Gallager68}, the sequence $n \mapsto n\tilde{R}_{\mu, p}(n, \eps)$ is subadditive. \cite{Gallager68} provides a proof for stationary stochastic process with finite alphabet, but it extends verbatim to our setting. Hence, we can define
	\[
	R_{\mu,p}(\eps)=\lim_{n \to \infty}R_{\mu,p}(n,\varepsilon) = \inf_{n \in \N}R_{\mu,p}(n,\varepsilon).
	\]
	
\end{defn}

\begin{rem}\label{rem:rate_distortion_properties}
	Similarly to \cite[Remark IV.3]{lindenstrauss_tsukamoto2017rate}, in Definition \ref{def:rate_distortion_function} it is enough to consider random vectors $Y$ taking only finitely many values. As $I(X;Y) \leq H(Y) < \infty$ for $Y$ taking finitely many values, we obtain that $R_{\mu,p}(\eps) < \infty$ for every $\eps>0$ (see also Remark \ref{rem:rd_function_finite}). Note that $R_{\mu, 1}(\eps) \leq R_{\mu,p}(\eps)$ for $p \in [1, \infty)$ and $\eps \geq 0$. For more details on the rate-distortion function see \cite{Gray11, Gallager68}.
\end{rem}
\subsection{Variational principle for metric mean dimension}
The following theorem is a variant of the variational principle for metric mean dimension for subshifts $\mS \subset [0,1]^\Z$. It is deduced from the original theorem in \cite{lindenstrauss_tsukamoto2017rate} - see Appendix \ref{app:var_prin}.
\begin{thm}\label{thm:var_prin_R_tilde} Let $\mS \subset [0,1]^\Z$ be a subshift. Then for $p \in  [1, \infty)$
	\begin{equation}
	\begin{split}
	\overline{\mdim}_{\mathrm{M}}(\mS) & =\limsup_{\varepsilon\to0}\frac{\sup_{\mu\in  \mP_{\sigma}(\mS)}R_{\mu, p}(\varepsilon)}{\log\frac{1}{\eps}}.
	\end{split}
	\end{equation}
\end{thm}

\section{Related works}\label{sec:related_works}

\numberwithin{thm}{section}

In this section we discuss results by other authors which aim at establishing analog compression rates in various settings, together with corresponding compression-decompression algorithms. The main difference with our work is that these results give bounds on compression rates for a \emph{fixed} distribution $\mu \in \mP_{\sigma}([0,1]^\Z)$ rather than all distributions supported by a given subshift. The bounds are given in terms of notions which can be seen as probabilistic notions of mean dimension.

In their pioneering article \cite{WV10} Wu and Verd\'u calculated or bounded from
below or above $\r_{\mathcal{{C}}-\mathcal{{D}}}(\mu,\eps)$ for
certain $\mC,\mD$ and $\mu\in\mathcal{{P}}_{\sigma}(\mathbb{{R}^{\mathbb{{N}}}})$. The notion which appears in the bounds is the \emph{Minkowski-dimension compression rate}, which (similarly to metric mean dimension and mean box dimension) is based on the geometrical notion of the box-counting (Minkowski) dimension.

\begin{defn}\label{d:measurable_mean_minkowski}\cite[Definition 10]{WV10} For a subshift $\mS\subset[0,1]^{\Z}$, invariant measure $\mu\in \mP_{\sigma}(\mS)$ and $0 \leq \eps<1$ define the \textbf{Minkowski-dimension compression rate} as
	\[
	R_{B}(\mu,\eps)= \limsup \limits_{n \to \infty}\ \inf \Big\{ \frac{\overline{\dim}_{B}(A)}{n} : A\subset[0,1]^{n},\ A \text{ - compact, } \mu(\pi_{n}^{-1}(A))\geq 1-\eps\Big\}.
	\]
\end{defn}
One can bound the Minkowski-dimension compression rate from above by mean box dimension. More precisely, the following inequalities hold for every subshift $\mS \subset [0,1]^\Z, \mu \in \mP_{\sigma}(\mS)$ and $0 \leq \eps<1$
\begin{equation}\label{e:RB_mbdim}
R_{B}(\mu,\eps)\leq R_{B}(\mu,0) \leq\mbdim(\mS).
\end{equation}

Similarly to Theorems \ref{thm:main_holder_low} and \ref{thm:main_holder_up}, Wu and Verd\'u give lower bounds for Borel-H\"{o}lder compression scheme and upper bounds for linear-H\"{o}lder compression scheme. Note that they work with the class $\mH_{\alpha}$ instead of $\mH^p_{L,\alpha}$, hence their results do not guarantee a uniform bound on the Lipschitz constant among the sequence of decoders.

\begin{thm}\label{thm:wu_verdu}\cite[Eq. (75) and Theorem 18]{WV10} 
	For $\mu \in \mP_{\sigma}[0,1]^{\Z},\ \eps \in (0,1)$ and $\alpha \in (0,1)$ the following holds:
	\[ \alpha R_B(\mu, \eps) \leq \r_{\mathcal{B}-\mathcal{H}_{\alpha}}(\mu, \eps) \leq \r_{\mathrm{LIN}-\mathcal{H}_{\alpha}}(\mu, \eps) \leq \frac{1}{1 - \alpha}R_B(\mu, \eps). \]
Consequently,
	\[\r_{\mathrm{LIN}-\mathcal{H}}(\mu, \eps) \leq R_B(\mu, \eps).\]
\end{thm}
Our Corollary \ref{cor:main_bounds_joined} can be seen as a uniform analog of the above result. Note that in Corollary \ref{cor:main_bounds_joined} the constant in the upper bound is worse: $\frac{2}{1 - \alpha}$ instead of $\frac{1}{1 - \alpha}$. This is a consequence of considering the zero-probability error case and, as shown in Example \ref{ex:sparse_LIN-H}, cannot be improved. Note also that the upper bound of Theorem \ref{thm:wu_verdu} together with (\ref{e:RB_mbdim}) imply
\[\sup \limits_{\mu \in \mP_{\sigma}(\mS)}\ \sup \limits_{\eps>0}\
\r_{\mathrm{LIN}-\mathcal{H}_{\alpha}}(\mu, \eps) \leq \frac{1}{1 - \alpha}\mbdim(\mS), \]
yet this is not enough to obtain a similar bound with $\mathcal{H}^p_{L,\alpha}$ replacing $\mathcal{H}_{\alpha}$, as this would require assuring that the Lipschitz constant $L$ is uniformly bounded among the sequence of decompressors achieving the upper bound in Theorem \ref{thm:wu_verdu}. Wu and Verd\'u proved in \cite{WV10} the lower bound only for Lipschitz decompressors. For the convenience of the reader, we include the proof for H\"{o}lder decompressors in Appendix \ref{app:wv_lower_bound} (see Proposition \ref{prop:rbh_below}).	The upper bound on $\r_{\mathrm{LIN}-\mathcal{H}_{\alpha}}(\mu, \eps)$ comes from minimizing $R$ in \cite[(172)]{WV10} for fixed $\beta$. Stronger result than the existence of linear compressor and H\"{o}lder decompressor was proven in \cite{SRAB17}, where it is shown that almost every linear transformation of rank large enough serves as a good compressor in this setting. Namely, the authors proved that for every $\eta > 0$, Lebesgue almost every matrix $A \in \R^{n \times k}$ with $k \geq (\frac{1}{1 - \alpha}R_B(\mu, \eps) + \eta)n$ admits an $\alpha$-H\"{o}lder decompressor $g : \R^k \to \R^n$ satisfying
	\[ \mu(\{x\in \mathcal{X}|\ g\circ A(x|_{0}^{n-1})\neq x|_{0}^{n-1}\})\leq\eps. \]
	For details see \cite[Subsection VIII]{SRAB17}. Similar results for signal separation have been obtained in \cite{stotz2013almost} and \cite{SRAB17}. 

\begin{rem} We assume the compressor functions to take values in the unit cube $[0,1]^k$, hence we assume an universal bound on signals after compression. Results in \cite{WV10} are stated with the compressor taking values in $\R^k$, but note that (since we consider compact spaces) composing with suitable affine transformations will give compressor functions with values in $[0,1]^k$ with the same H\"{o}lder exponent $\alpha$ of the decompressor and possibly different (but not arbitrary) Lipschitz constant $L$.
\end{rem}

Let us discuss now results of \cite{CompressionCS17, jalali2017universal}. This requires introducing one more type of compression rate.

\begin{defn} \label{def:Prob_ratio is achievable}
Let $\mu\in  \mP_{\sigma}([0,1]^\Z)$. Let $\mathcal{{C}},\mathcal{{D}} \subset \{ f : [0,1]^n \to [0,1]^k : n, k \in \N \}$ be regularity classes with respect to the norm $\| \cdot \|_p$. For $n \in \N$ and $\eps \geq 0,\ \delta \geq 0$, the $\mathcal{C}-\mathcal{D}$ \textbf{probability analog compression rate $\r_{\mathcal{C}-\mathcal{D}}^{P,p}(\mu,\eps,n, \delta)\geq0$} of $\mu$ with $n$-block error probability $\delta$ at scale $\eps$ is the infimum of
	\[\frac{k}{n},\]
	where $k$ runs over all natural numbers such that there exist maps
	\[
	f:[0,1]^{n}\rightarrow[0,1]^{k},\ f\in\mathcal{{C}}\ \text{ and }\ g:[0,1]^{k}\rightarrow[0,1]^{n},\ g\in\mathcal{{D}}
	\]
	with
	\[
	\mu( \{ x \in [0,1]^\Z : \|x|_0^{n-1} - g \circ f (x|_0^{n-1})\|_p \geq \eps \}) \leq \delta.
	\]
	Define further:
	\[
	\r_{\mathcal{C}-\mathcal{D}}^{P, p}(\mu,\eps, n)=\lim_{\delta \to 0}\ \r_{\mathcal{C}-\mathcal{D}}^{P,p}(\mu,\eps,n, \delta),
	\]
	\[
	\r_{\mathcal{C}-\mathcal{D}}^{P,p}(\mu,\eps)=\limsup_{n\rightarrow\infty}\ \r_{\mathcal{C}-\mathcal{D}}^{P,p}(\mu,\eps,n).
	\]
\end{defn}

Note that if $\mathcal{C}, \mathcal{D}$ are regularity classes with respect to the norm $\| \cdot \|_p$ for $p \in [1,\infty)$ , then for any $\eps' > \eps \geq 0$
\begin{equation} \r_{\mathcal{C}-\mathcal{D}}^{L^p}(\mu,\eps') \leq \r_{\mathcal{C}-\mathcal{D}}^{P,p}(\mu,\eps).
\end{equation}
This follows from the observation that condition $\mu( \{ x \in [0,1]^\Z : \|x|_0^{n-1} - g \circ f (x|_0^{n-1})\| \geq \eps \}) \leq \delta$ implies $\int \limits_{[0,1]^n} \| x - g \circ f (x)\|^p d (\pi_n)_* \mu \leq \delta + \eps^p \leq (\eps')^p$ for $\delta$ small enough.

The results of \cite{CompressionCS17} are given in terms of the \emph{rate-distortion dimension} of a measure $\mu \in \mP_{\sigma}([0,1]^\Z)$, defined as
\begin{equation}\label{eq:rdim_def} \overline{\dim}_{R}(\mu) = \limsup \limits_{\eps \to 0} \frac{R_{\mu, 2}(\eps)}{\log \frac{1}{\eps}},
\end{equation}
where $R_{\mu, 2}(\eps)$ is the $L^2$ rate-distortion function (see Definition \ref{def:rate_distortion_function}). In \cite{CompressionCS17} results on analog compression in terms of the rate-distortion dimension are obtained using the techniques of compressed sensing. The authors consider a linear compression algorithm in which decompression (given via a suitable minimization problem) is based on a sequence of compression codes with distortion approaching zero. These codes depend on the process and are assumed to be known a priori. In our notation, the authors obtained the following result.

\begin{thm}\label{thm:CompressionCS} \cite[Corollary 2]{CompressionCS17}
	Let $\mu \in \mP_{\sigma}([0,1]^\Z)$. Then
	\[ \sup \limits_{\eps > 0}\ \r^{P,2}_{\LIN - \mB}(\mu, \eps) \leq \overline{\dim}_{R}(\mu). \]
\end{thm}

\begin{rem}
	In applications, the measure governing the source is not always known, hence one may not have an access to the compression codes required in \cite{CompressionCS17}. A universal algorithm for a certain class of processes was proposed by Jalali and Poor in \cite{jalali2017universal}. For a measure $\mu \in \mP_{\sigma}([0,1]^\Z)$ define its \emph{upper information dimension} as
\[\overline{d}_0(\mu) = \lim_{n\rightarrow\infty} \frac{\uid((\pi_n)_* \mu)}{n}, \]
where $\uid((\pi_n)_* \mu)$ is the upper R\'{e}nyi information dimension of the measure $(\pi_n)_* \mu$ on $[0,1]^n$ (see \cite[Section IV]{jalali2017universal} for details). In terms of analog compression rates, \cite[Theorems 7,8]{jalali2017universal} give the following bound (for the definition of $\psi^*$-mixing see \cite[Definition 3]{jalali2017universal}):
	
	\begin{equation}\label{eq:jalali_poor} \sup \limits_{\eps > 0}\ \r^{P,2}_{\LIN - \mB}(\mu, \eps) \leq \overline{d}_0(\mu) \text{ if } \mu \in \mP_{\sigma}([0,1]^\Z) \text{ is }\psi^*\text{-mixing}.
	\end{equation}
	
	Theorem \ref{thm:CompressionCS} is stronger than (\ref{eq:jalali_poor}). Indeed, in general $ \overline{\dim}_{R, 2} \leq \overline{d}_0(\mu)$ (see \cite[Theorem 14]{GK18}) and the equality can be strict (see \cite[Example 1]{GK17}). Also, $\psi^*$-mixing is a quite restrictive assumption. However, inequality (\ref{eq:jalali_poor}) does not reflect the full substance of the original results of \cite{jalali2017universal}. Jalali and Poor proved more than merely existence of suitable linear compressors. More precisely, they proved that for any $\eta>0$, if $(X_n)_{n \in \Z}$ is a $\psi^*$-mixing stochastic process with distribution $\mu$ and $A_n \in \R^{n \times m_n}$ are independent random matrices with entries drawn i.i.d according to $\mathcal{N}(0, 1)$ and independently from $(X_n)_{n \in \Z}$ with $\frac{m_n}{n} \geq (1+ \eta) \overline{d}_0(\mu)$, then
	\[\|X|_0^{n-1} - g_n \circ A_n (X|_0^{n-1}) \|_2 \overset{n \to \infty}{\longrightarrow} 0 \text{ in probability}\]
	for some explicitly defined Borel functions $g_n : \R^{m_n} \to \R^n$ (depending only on $A_n$). Hence, for such a random sequence of matrices, the expected value
	\[ \E \mu( \{ x \in \mX : \|x|_{0}^{n-1} - g_n \circ A_n (x|_{0}^{n-1})\|_2 \geq \eps \}) \]
	tends to zero as $n \to \infty$ for any $\psi^*$-mixing measure $\mu \in \mP_{\sigma}([0,1]^{\Z})$. Inequality (\ref{eq:jalali_poor}) follows from this, since for any $\delta >0$ and $n$ large enough, there exists $A_n \in \R^{n \times m_n}$ satisfying
	\begin{equation}\label{e:jalali_poor} \mu( \{ x \in \mX : \|x|_{0}^{n-1} - g_n \circ A_n (x|_{0}^{n-1})\|_2 \geq \eps \}) \leq \delta.
	\end{equation}
	Decompressors $g_n$ are defined via a certain minimization problem (which makes the decompression algorithm implementable, though not efficient (cf. \cite[Remark 3]{jalali2017universal})). The authors proved also that, in a certain setting, such a compression scheme is robust to noise. More precisely, they proved (\cite[Theorems 9 and 10]{jalali2017universal})
	\[\|X|_0^{n-1} - g_n ( A_n X|_0^{n-1} + Z_n) \|_2 \overset{n \to \infty}{\longrightarrow} 0 \text{ in probability}\]
	as long as $A_n \in \R^{n \times m_n}$ has i.i.d entries distributed according to $\mathcal{N}(0, \frac{1}{n})$ and $(Z_n)_{n \in \N_0}$ is a stochastic process converging to zero in probability fast enough. Since the functions $g_n$, defined by a certain minimization problem, take only finitely many values, they cannot be taken to be continuous. The strength of the result is the universality of the compression scheme, which is designed without any prior knowledge of the distribution $\mu$: a random Gaussian matrix will serve as a good compressor as long as the rate is at least $\overline{d}_0(\mu)$ and the decompressor is explicit. However, it does not follow that one can choose a sequence of matrices $A_n$ satisfying (\ref{e:jalali_poor}) for all $\psi^*$-mixing measures $\mu$ with $\overline{d}_0(\mu) \leq d$ for some $d \in [0,1]$. \\
	
\end{rem}
We are also able to prove a lower bound on $\sup \limits_{\eps>0}\ \r^{L^2}_{\mathcal{B} - \mathcal{H}^2_{L, \alpha}}(\mu, \eps)$ for a fixed measure $\mu$ in terms of rate-distortion dimension $\overline{\dim}_R(\mu)$ (see Corollary \ref{cor:rate_lower_fixed_measure}). It turns out that $\overline{\dim}_{R}$ is equal to the information dimension rate $\overline{d}(\mu)$ introduced in \cite{GK17}. Under certain assumptions both of them coincide with $\overline{d}_0(\mu)$. See Section \cite[Section V]{GK18} for a comprehensive discussion and \cite[Section VI]{GK18} for more on the operational meaning of the information dimensions $\overline{d}(\mu)$ and $\overline{d}_0(\mu)$.

\section{Proofs of the main results}\label{sec:proofs}
\numberwithin{thm}{subsection}
\subsection{Proof of Theorem \ref{thm:main_holder_low}}\label{sec:lower}
Theorem \ref{thm:main_holder_low} is a direct consequence of Theorem \ref{thm:var_prin_R_tilde}, inequality (\ref{e:rate_comparision}) and Theorem \ref{thm:main} below. The latter one is of independent interest, as it gives a lower bound for compression rates $\r^{L^p}_{\mathcal{{B}}-\mathcal{{H}}^p_{L,\alpha}}(\mu,\eps)$ and $\r_{\mathcal{{B}}-\mathcal{{H}}_{L,\alpha}^\infty}(\mu,\eps)$ for fixed $\mu \in \mP_{\sigma}(\mS)$ and $\eps > 0$ (see also Corollary \ref{cor:rate_lower_fixed_measure} below).

\begin{thm}\label{thm:main}
	Let $\mS \subset [0,1]^\Z$ be closed and shift-invariant.  The following
	holds for $\mu\in  \mP_{\sigma}(\mS),\ 0<\alpha\leq1,\ L>0$ and $p \in [1, \infty)$:	
	\[
	\frac{R_{\mu,p}((\frac{L^p}{2^{p \alpha}}+\eps^{p(1-\alpha)})^{\frac{1}{p}}\eps^{\alpha})}{\log(\lceil\frac{1}{\eps}\rceil)}\leq \r^{L^p}_{\mathcal{{B}}-\mathcal{{H}}_{L,\alpha}^p}(\mu,\eps).
	\]
Similarly, for $p=\infty$ the following holds

\[
\frac{R_{\mu, 1}((\frac{L}{2^{ \alpha}}+\eps^{(1-\alpha)})\eps^{\alpha})}{\log(\lceil\frac{1}{\eps}\rceil)}\leq \r_{\mathcal{{B}}-\mathcal{{H}}_{L,\alpha}^\infty}(\mu,\eps).
\]
\end{thm}

\begin{proof}
	Let us begin with proving the first assertion (for $p < \infty$). Fix $\delta,\eps>0$. By the definition of $\r^{L^p}_{\mathcal{{B}}-\mathcal{H}^p_{L, \alpha}}(\mu,\eps)$, one may find $k,n\in\mathbb{N}$ with
	$\frac{k}{n}\leq \r_{\mathcal{{B}}-\mathcal{{H}}^p_{L,\alpha}}^{L^p}(\mu,\eps)+\delta$ and functions
	$f:[0,1]^{n}\rightarrow[0,1]^{k},\ f \in \mathcal{B}$,
	$g:[0,1]^{k}\rightarrow[0,1]^{n},\ g \in \mathcal{H}^p_{L, \alpha}$ such that
	\begin{equation}\label{eq:Lp_good_compressor} \int \limits_{[0,1]^n} \|x - g \circ f (x)  \|_p^p d (\pi_n)_* \mu(x) \leq \eps^p. \end{equation}
	Regularly partition $[0,1]^{k}$ into $\lceil\frac{1}{\eps}\rceil^{k}$ cubes of side
	$\lceil\frac{1}{\eps}\rceil^{-1}$ Borel-wise (thus every point
	in $[0,1]^{k}$ is associated by a Borel selector to exactly one cube
	with edge length $\lceil\frac{1}{\eps}\rceil^{-1}$ which contains
	it) and let $c:[0,1]^{k}\rightarrow F$ associate to each point the
	center of its cube. Note that $\#F=\lceil\frac{1}{\eps}\rceil^{k}$
	and $||x-c(x)||_p \leq||x-c(x)||_{\infty}\leq\frac{\eps}{2}$ for all $x\in[0,1]^{k}$.
	Define $Y:[0,1]^{n}\rightarrow[0,1]^{n}$
	by
	\[Y(p)=g(c(f(p))). \]
	and $X:[0,1]^{n} \to [0,1]^{n}$ by $X = \id$. This gives a pair of random vectors on the probability space $([0,1]^n, (\pi_n)_* \mu)$.
	We now estimate:
	\[
	\mathbb{E_{\mu}}\|(X_{0},X_{1},\dots,X_{n-1})-(Y_{0},Y_{1},\dots,Y_{n-1})\|_p^p = \int \limits_{[0,1]^n} \|x - g \circ c \circ f (x) \|_p^p d (\pi_n)_* \mu(x) \leq \]
	
	\[ \leq \int \limits_{[0,1]^n} \|x - g  \circ f (x) \|_p^p d (\pi_n)_* \mu(x) +  \int \limits_{[0,1]^n} \|g \circ f(x) - g \circ c \circ f (x) \|_p^p d (\pi_n)_* \mu(x) \leq \]
	\[ \leq \eps^p + \int \limits_{[0,1]^n} L^{p} \|f(x) - c \circ f(x)\|_p^{p\alpha}d (\pi_n)_* \mu(x) \leq \eps^p + L^p \frac{\eps^{p \alpha}}{2 ^ {p \alpha}}. \]
	This implies that we have found $X$ and $Y$ obeying condition (\ref{eq: distortion condition}) at scale $(\eps^p + L^p \frac{\eps^{p \alpha}}{2 ^ {p \alpha}})^{\frac{1}{p}} = (\frac{L^p}{2^{p \alpha}}+\eps^{p(1-\alpha)})^{\frac{1}{p}}\eps^{\alpha} $, hence
	\[
	R_{\mu,p}((\frac{L^p}{2^{p \alpha}}+\eps^{p(1-\alpha)})^{\frac{1}{p}}\eps^{\alpha})\leq\frac{1}{n}I(X;Y)\leq\frac{1}{n}H(Y)\leq \frac{\log(\lceil\frac{1}{\eps}\rceil^{k})}{n}=\frac{k\log(\lceil\frac{1}{\eps}\rceil)}{n} \leq \log(\lceil\frac{1}{\eps}\rceil)(\r_{\mathcal{{B}}-\mathcal{{H}}^p_{L, \alpha}}^{L^p}(\mu,\eps)+\delta).
	\]
	We end the proof by dividing by $\log(\lceil\frac{1}{\eps}\rceil)$.

For $p = \infty$ one considers the same construction, but functions $f$ and $g$ belong now to $\mB$ and $\mH_{L,\alpha}^\infty$ respectively and satisfy
	\[ \mu \big( \{x\in \mathcal{X}|\ g\circ f(x|_{0}^{n-1})\neq x|_{0}^{n-1}\} \big) \leq \eps \]
	instead of (\ref{eq:Lp_good_compressor}). Defining $X$ and $Y$ as before, we have
	\[
	\mathbb{E}_{\mu}\|(X_{0},X_{1},\dots,X_{n-1})-(Y_{0},Y_{1},\dots,Y_{n-1})\|_1 \leq \]
\[ \leq \int \limits_{[0,1]^n} \|x - g  \circ f (x) \|_\infty d (\pi_n)_* \mu(x)  + \int \limits_{[0,1]^n} \|g \circ f(x) - g \circ c \circ f (x) \|_\infty d (\pi_n)_* \mu(x) \leq \]
	\[ \leq \eps + \int \limits_{[0,1]^n} L \|f(x) - c \circ f(x)\|_{\infty}^{\alpha}d (\pi_n)_* \mu(x) \leq \eps + L \frac{\eps^{\alpha}}{2^{\alpha}}. \]
	This implies
	\[
	R_{\mu,1}((\frac{L}{2^{ \alpha}}+\eps^{1-\alpha})\eps^{\alpha})\leq\frac{1}{n}I(X;Y)\leq\frac{1}{n}H(Y)\leq \frac{\log(\lceil\frac{1}{\eps}\rceil^{k})}{n}=\frac{k\log(\lceil\frac{1}{\eps}\rceil)}{n}\leq \log(\lceil\frac{1}{\eps}\rceil)(\r_{\mathcal{{B}}-\mathcal{{H}}^{\infty}_{L, \alpha}}(\mu,\eps)+\delta).\]
\end{proof}

\begin{cor}\label{cor:rate_lower_fixed_measure}
Let $\mu \in \mP_{\sigma}([0,1]^\Z),\ \alpha \in (0,1],\ L>0$. Then (recall (\ref{eq:rdim_def}) for the definition of $\overline{\dim}_{R} (\mu)$)	\begin{equation}\label{e:rate_lower_fixed_measure} \alpha \overline{\dim}_{R} (\mu)  \leq \r^{L^2}_{\mathcal{B} - \mathcal{H}^2_{L, \alpha}}(\mu).
\end{equation}
\end{cor}
\begin{rem}
	Note that taking supremum over $\mu \in \mP_{\sigma}(\mS)$ in (\ref{e:rate_lower_fixed_measure}) yields
	\[ \sup \limits_{\mu \in \mP_{\sigma}(\mS)} \alpha\overline{\dim}_{R} (\mu) \leq \sup \limits_{\eps>0} \sup \limits_{\mu \in \mP_{\sigma}(\mS)}\ \r^{L^2}_{\mathcal{B} - \mathcal{H}^2_{L, \alpha}}(\mu, \eps).\] This, however, does not imply Theorem \ref{thm:main_holder_low}, since there are systems with
	\[ \sup \limits_{\mu \in \mP_{\sigma}(\mS)} \limsup \limits_{\eps \to 0} \frac{R_{\mu,2}(\eps)}{\log \frac{1}{\eps}} < \ummdim(\mS)  \]
	(cf. \cite[Section VIII]{lindenstrauss_tsukamoto2017rate}). In other words, to obtain Theorem \ref{thm:main_holder_low}, the supremum over $\mu \in \mP_{\sigma}(\mS)$ has to be taken \emph{before} passing with $\eps$ to zero, as one cannot exchange these two operations in the variational principle.
\end{rem}

\subsection{Proof of Theorem \ref{thm:main_holder_up}}\label{sec:upper}
In order to develop the upper bound in Theorem \ref{thm:main_holder_up}, we will make use of the embedding theorem for the upper box-counting dimension. Below we present a corollary of \cite[Theorem 4.3]{Rob11} (with proof attributed to \cite{HK99}, based on an earlier result in \cite{benartzi1993holder}).

\begin{thm}\label{thm:box_embedding}
	Let $A$ be a compact subset of $\R^n$ and fix $p \in [1,\infty]$. If $k > 2 \udim(A)$, then given any $\alpha$ with
	\[ 0 < \alpha < 1 - \frac{2\udim(A)}{k}, \]
	for Lebesgue almost every linear map $F \in \mathrm{LIN}(\R^n, \R^k) \simeq \R^{nk} $ there exists $L = L(F,p)$ such that
	\[ \|x - y\|_p \leq L \|F(x) - F(y)\|_p^{\alpha}\ \text{ for all } x,\ y \in A. \]
	In particular, $F$ is one-to-one on $A$ with inverse which is $(L,\alpha)$-H\"{o}lder in the norm $\| \cdot \|_p$. 
\end{thm}
\begin{rem} The only manner in which choosing the norm $\| \cdot \|_p$ influences the above theorem is the dependence of the Lipschitz constant $L$ on the norm. 
\end{rem}
See also \cite{BGS18} for an almost sure embedding theorem for Hausdorff dimension. However, it does not provide  H\"{o}lder inverse and argues that it cannot exist in general in the probabilistic context. Similar results for the modified lower box-counting dimension have been obtained also in \cite{LAC}.

We will use now Theorem \ref{thm:box_embedding} to prove Theorem \ref{thm:main_holder_up}. Let $\mS \subset [0,1]^\Z$ be a subshift. We shall prove
\begin{equation}\label{e:main_upper}
\inf \limits_{L>0}\ \r_{\mathrm{LIN}-\mathcal{{H}}^p_{L,\alpha}}(\mS,0) \leq  \frac{2}{1-\alpha}\umbdim(\mS)
\end{equation}
for every $0<\alpha < 1$ and $p \in [1, \infty]$. Fix $\eta > 0$. By Lemma \ref{l:R0_subadd} (applied with $\mathcal{C} = \mathrm{LIN}$) it suffices to prove that there exists $L>0$ and $n \in \N$ such that
\begin{equation}\label{eq:suff_upper} \r_{\mathrm{LIN}-\mathcal{{H}}^p_{L,\alpha}}(\mS,0,n) \leq \frac{2}{1 - \alpha}\umbdim(\mS) + \eta.
\end{equation}
	Fix $n \in \N$ large enough to obtain $\frac{1}{n} \leq \frac{\eta}{4}$ and
\begin{equation}\label{eq:bdim_approx_mbdim}
\frac{\udim(\pi_n(\mS))}{n} \leq \umbdim(\mS) + \frac{\eta(1-\alpha)}{4}.
\end{equation}
	By Theorem \ref{thm:box_embedding}, if $\alpha \in (0,1)$ and $k \in \N,\ k > 2 \udim(\pi_n(\mS))$ fulfil
	\begin{equation}\label{e:k_alpha}
	\alpha < 1 - \frac{2\udim(\pi_n(\mS))}{k},
	\end{equation}
	then there exists $f : [0,1]^n \to [0,1]^k$ in $\mathrm{LIN}$ and $L>0$ such that $f|_{\pi_n(\mS)}$ is injective with $(L, \alpha)$-H\"{o}lder inverse in the norm $\| \cdot \|_p$. According to the Remark \ref{rem:extension_theorems}, $f^{-1}: f(\pi_n(\mS) \to [0,1]^n$ can be extended to function $g : [0,1]^k \to [0,1]^n$ belonging to $\mH^p_{L', \alpha}$ for some $L' \geq L$.  We obtain that $\r_{\mathrm{LIN}-\mathcal{{H}}_{L',\alpha}}(\mS,0,n) \leq \frac{k}{n}$ for every such $k$. For fixed $\alpha \in (0,1)$, condition (\ref{e:k_alpha}) is satisfied for $k = \lceil \frac{2\udim(\pi_n(\mS))}{1-\alpha} \rceil + 1$, hence by (\ref{eq:bdim_approx_mbdim})
	\[ \r_{\mathrm{LIN}-\mathcal{{H}}^p_{L',\alpha}}(\mu,0,n) \leq \frac{1}{n} \Big(\lceil \frac{2\udim(\pi_n(\mS))}{1-\alpha} \rceil + 1 \Big)  \leq \frac{1}{n}\Big( \frac{2\udim(\pi_n(\mS)) }{1-\alpha} +2 \Big) \leq \]
	\[ \leq   \frac{2}{1-\alpha}\Big(\umbdim(\mS) + \frac{\eta(1-\alpha)}{4}\Big) + \frac{2}{n}  \leq \frac{2}{1-\alpha}\umbdim(\mS) + \eta. \]
Therefore $L'$ and $n$ satisfy (\ref{eq:suff_upper}). This concludes the proof of (\ref{e:main_upper}) and Theorem \ref{thm:main_holder_up}.

\begin{rem}
	Example \ref{ex:two_symbols_shift} shows that one cannot improve Theorem \ref{thm:main_holder_up} by claiming that there exists finite $L>0$ such that $\r_{\mathrm{LIN}-\mathcal{{H}}^1_{L,\alpha}}(\mS,0) \leq \frac{2}{1-\alpha}\umbdim(\mS)$. Indeed, for $\mS = \{0,1\}^\Z$ we have $\umbdim(\mS) = 0$, yet for $L>0$ inequality
	\[ \r_{\mathrm{LIN}-\mathcal{{H}}^1_{L,\alpha}}(\mS,0) \geq  \sup \limits_{\varepsilon > 0} \ \sup \limits_{\mu \in  \mP_{\sigma}(\mS)} \r^{L^1}_{\mathcal{B} - \mathcal{H}^1_{L, \alpha}}(\mu, \eps) > 0\]
	holds.
\end{rem}

\subsection{Proof of Proposition \ref{prop:peano_up}}\label{sec:peano}
The proof of Proposition \ref{prop:peano_up} is based on an existence of surjective H\"{o}lder maps between unit cubes. It is well known that there exists a surjective $\frac{1}{n}$-H\"{o}lder map  from $[0,1]$ onto $[0,1]^n$ (see \cite[Thm. 4.55]{Milne80}). This is a generalization of the classical Peano curve construction (see e.g. \cite{Sagan94}). We will use a similar construction to obtain the following proposition:

\begin{prop}\label{prop:holder_peano_cubes} For every $n, k \in \N,\ k \leq n,\ p \in [1,\infty]$ and $\alpha < \frac{k}{n}$ there exist maps $f : [0,1]^n \to [0,1]^k$ and $g : [0,1]^k \to [0,1]^n$ such that $f \in \mathcal{B},\ g \in \mathcal{H}_\alpha$ and $g \circ f (x) = x$ for every $x \in [0,1]^n$.
\end{prop}

\begin{proof}
As $\alpha < \frac{k}{n}$, there exists a surjective and $\alpha$-H\"{o}lder map $g : [0,1]^k \to [0,1]^n$ (see \cite[Comments to Problem 1988-5]{ArnoldProblems} or \cite{Shchepin2010}). Let $C([0,1]^k)$ be the space of closed non-empty subsets of $[0,1]^k$ equipped with the Vietoris topology. Recall that its Borel space $B(C([0,1]^k))$	is the Effros Borel space of $[0,1]^k$, i.e. the $\sigma$-algebra generated by the sets of the form $\{ F \in C([0,1]^k) : F \cap U \neq \emptyset \}$ for open $U \subset [0,1]^k$ (\cite[12.7]{K95}). Note that $M : [0,1]^n \to C([0,1]^k)$ given as $M(x) = g^{-1}(\{x\})$ is Borel, as for every open $U \subset [0,1]^k$ there is \[M^{-1}(\{ F \in C([0,1]^k) : F \cap U \neq \emptyset \}) = \big\{ x \in [0,1]^n : g^{-1}(\{x\}) \in \{ F \in C([0,1]^k) : F \cap U \neq \emptyset \} \big\} = \]
\[ =  \{ x \in [0,1]^n : g^{-1}(\{x\})  \cap U \neq \emptyset \} = g(U). \] The latter set is Borel as $g$ is continuous and $U$ is $F_{\sigma}$ an open subset of a compact space $[0,1]^k$. By the Kuratowski-Ryll-Nardzewski selection theorem (\cite[12.13]{K95}) there is a Borel selector $s:C([0,1]^k) \to [0,1]^k$ (i.e. a Borel map satisfying $s(A) \in A$ for every $A \in C([0,1]^k)$). Let $f : [0,1]^n \to [0,1]^k$ be given as $f = s \circ M$. Then $f$ is Borel and $g(f(x)) = x$ as $f(x) = s (g^{-1}(\{x\})) \in g^{-1}(\{x\})$.
\end{proof}

Let us prove now Proposition \ref{prop:peano_up}. Let $\mS \subset [0,1]^\Z$ be a subshift. Fix $0 < \alpha \leq 1$ and $p \in [1, \infty]$. We shall prove
\begin{equation}\label{eq:peano_up} \inf \limits_{L>0}\ \r_{\mathcal{B} - \mathcal{H}_{L, \alpha}^{p}} (\mS, 0) \leq \alpha.
\end{equation}
The inequality is clear for $\alpha=1$, hence we can assume $\alpha < 1$. Fix arbitrary $\eta > 0$. By Lemma \ref{l:R0_subadd} (applied with $\mathcal{C} = \mathcal{B}$)
	\[\inf \limits_{L>0}\ \r_{\mathcal{B} - \mathcal{H}^p_{L, \alpha}}(\mS, 0) = \inf \limits_{L>0}\ \inf \limits_{n \in \N}\  \r_{\mathcal{B} - \mathcal{H}^p_{L, \alpha}}(\mS, 0, n),\]
	hence it is enough to prove that there exists $n \in \N$ and $L>0$ with $\r_{\mathcal{B} - \mathcal{H}_{L, \alpha}^p}(\mS, 0, n) \leq \alpha+\eta$. Let $k,n \in \N$ be such that $k \leq n$ and $\alpha < \frac{k}{n} \leq \alpha + \eta$. By Proposition \ref{prop:holder_peano_cubes} there exists $f : [0,1]^n \to [0,1]^k$ and $g : [0,1]^k \to [0,1]^n$ such that $f \in \mathcal{B},\ g \in \mathcal{H}_\alpha$ and $g \circ f (x) = x$ for every $x \in [0,1]^n$. Therefore there exists $L>0$ such that
	\[ \r_{\mathcal{B} - \mathcal{H}^p_{L, \alpha}}(\mS, 0, n) \leq \frac{k}{n} \leq \alpha+\eta. \]
As $\eta$ was arbitrary, (\ref{eq:peano_up}) is proved.

\numberwithin{thm}{section}
\section{Concluding remarks}\label{sec:concluding_remarks}
In this paper, we have considered the problem of finding worst-case compression rates, which are sufficient for \emph{all} stationary stochastic processes taking values in a given set $\mS \subset [0,1]^\Z$, under certain constraints of the compression and decompression process. In the case of Borel compression and $(L,\alpha)$-H\"{o}lder decompression (with both parameters fixed among the sequence of decoders), we have obtained a lower bound in terms of the metric mean dimension $\mmdim(\mS)$ of the dynamical system $(\mS, \mathrm{shift})$:
\begin{equation}\label{eq:conclude_main} \alpha \ummdim(\mS) \leq \sup \limits_{\mu \in \mP_{\sigma}(\mS)}\ \sup \limits_{\eps>0}\ \r_{\mB - \mH_{L,\alpha}^p}^{L^p}(\mS , \eps).
\end{equation}
Intuitively, this result states that for a given set of trajectories $\mS$, one can always find a stationary stochastic process supported in $\mS$, which cannot be compressed at a better rate than $\alpha \ummdim(\mS)$. This can be applied whenever geometric information on the signal is given (i.e. one knows $\mS$), but its statistics are unknown.

We have obtained also two zero-probability error upper bounds on the compression rates. The first one considers linear compression and H\"{o}lder decompression:
\[ \inf \limits_{L>0}\ \r_{\mathrm{LIN}-\mathcal{{H}}_{L,\alpha}^p}(\mS,0) \leq  \frac{2}{1-\alpha}\mbdim(\mS).\]
It is obtained by applying a finite-dimensional embedding theorem for the upper box-counting dimension. By considering $(N,K)$-sparse subshifts, we have shown that in general the constant $\frac{2}{1-\alpha}\mbdim(\mS)$ above cannot be improved in the zero-probability error case. The second upper bound considers Borel compression and H\"{o}lder decompression:
\[ \inf \limits_{L>0}\ \r_{\mathcal{B} - \mathcal{H}_{L, \alpha}^{p}} (\mS, 0) \leq \alpha. \]
The proof employs constructions of surjective H\"{o}lder maps between unit cubes akin to the Peano curve construction. This type of construction would not be possible for more stringent regularity requirements for the compressor.

This paper introduces notions and techniques from the theory of dynamical systems to the study of analog compression rates. Its main tool is the variational principle of Lindenstrauss and Tsukamoto \cite{lindenstrauss_tsukamoto2017rate}, which had previously established a link between ergodic theory and mean dimension theory.

As our main example, we have considered a parametrized family of sparse signal subshifts for which we proved that the lower bound (\ref{eq:conclude_main}) is in fact an equality (after taking infimum over $L>0$). Whenever equality holds in (\ref{eq:conclude_main}), it can be seen an operational characterization of metric mean dimension, as well as an answer to the given compression problem. It is therefore desirable to answer the following:

\begin{problem}\label{prob:equlatiy}
Under what conditions (for which subshifts) does equality hold in (\ref{eq:conclude_main}) (possibly after taking $\inf \limits_{L>0}$)?
\end{problem}

\appendix
\numberwithin{thm}{section}
\section{Metric mean dimension in dynamical systems}\label{app:mmdim}

In this section we present metric mean dimension in its original setting - the theory of dynamical systems. Recall that by a \emph{dynamical system} we understand triple $(\mX, \rho, T)$ consisting of a compact metric space $(\mX, \rho)$ and a homeomorphism $T : \mX \to \mX$. First, we introduce the notion of complexity of the system at scale $\eps > 0$.

\begin{defn}\label{d:mmdim}
\label{def:d_n}Let $(\mX, \rho)$ be a compact metric space and let $T:\mathcal{X}\to \mathcal{X}$
	be a homeomorphism. For $n\in\N$ define a metric $\rho_{n}$ on $\mathcal{X}$
	by $\rho_{n}(x,y)=\max \limits_{0\leq k<n}\rho(T^{k}x,T^{k}y)$.
	The upper and lower \textbf{metric mean dimensions} of the system $(\mX, T, \rho)$ are defined as
	\[
	\overline{\mdim}_{M}(\mathcal{X},T,d)=\limsup_{\eps\to0} \lim\limits_{n\to\infty}\frac{\log\#(\mathcal{X},\rho_{n},\eps)}{n\log\frac{1}{\eps}}
	\]
and	
	\[
	\underline{\mdim}_{M}(\mathcal{X},T,d)=\liminf_{\eps\to0} \lim\limits_{n\to\infty}\frac{\log\#(\mathcal{X},\rho_{n},\eps)}{n\log\frac{1}{\eps}},
	\]
where $\#(\mX, \rho_n, \eps)$ is the $\eps$-covering number of $\mX$ with respect to the metric $\rho_n$ (see Definition \ref{def:hashtag}). The limit with respect to $n$ exists due to the subadditivity of the sequence $n\mapsto\log\#(\mathcal{X},\rho_{n},\eps)$. If the upper and lower limits coincide, then we call its common value the \textbf{metric mean dimension} of $(\mX, T, \rho)$ and denote it by $\mmdim(\mX, T, \rho)$.
\end{defn}
\begin{rem}
Metric mean dimension is an invariant for bi-Lipshitz isomorphisms. Precisely, if $(\mX, T, \rho_1)$ and $(\mathcal{Y}, S, \rho_2)$ are dynamical systems and $\Phi : \mX \to \mathcal{Y}$ is a bi-Lipshitz bijection (i.e. a Lipschitz map with a Lipschitz inverse) which is equivariant (i.e. it satisfies $\Phi \circ T = S \circ \Phi$), then $\ummdim(\mX, T, \rho_1) = \ummdim(\mathcal{Y}, S, \rho_2)$.
\end{rem}

A topological version of mean dimension for actions of amenable groups was introduced by Gromov in \cite{G}. This invariant of topological dynamical systems was used by Lindenstrauss and Weiss in \cite{LW00} to answer a long standing question in topological dynamics: does every minimal\footnotemark \footnotetext{A system $(\mX, \rho, T)$ is called \emph{minimal} if for every $x \in \mX$, the orbit $\{ T^n x : n \in \Z \}$ is dense in $\mX$.} topological dynamical system embed into $([0,1]^\Z, \sigma)$? The answer is negative, since any system embeddable in $([0,1]^\Z, \sigma)$ has topological mean dimension at most one and in \cite{LW00} a minimal system with mean dimension strictly greater than one was constructed. Moreover, it was proved in \cite{GutTsu16}, that any minimal system with topological mean dimension strictly smaller than $\frac{D}{2}$ is embeddable into $(([0,1]^D)^\Z, \sigma)$. This constant is known to be optimal (cf. \cite{LT12}). One of the main tools in the proof is a variant of the Whittaker-Nyquist-Kotelnikov-Shannon sampling theorem, which indicates a connection between mean dimension theory and signal processing. For similar results for $\Z^k$ actions see \cite{gutman2017application} and \cite{GutLinTsu15}. For more on mean topological dimension see \cite{coornaert2015topological}. Metric mean dimension was introduced in \cite{LW00} and proved to be, when calculated with respect to any compatible metric, an upper bound for the topological mean dimension. It was recently successfully used in \cite{TsBrody18} and \cite{TsYangMills18} to obtain formulas for mean dimension of dynamical systems arising from geometric analysis.

If $\mS\subset[0,1]{}^{\mathbb{{Z}}}$ is a subshift, then Definition \ref{d:mmdim} with the dynamics given by the shift
\[ \sigma: \mS \to \mS,\ \sigma((x_{i})_{i=-\infty}^{\infty})=(x_{i+1})_{i=-\infty}^{\infty} \] and the metric
\[ \tau(x,y)=\sum \limits_{i=-\infty}^{\infty}\frac{1}{2^{|i|}}|x_{i}-y_{i}|,\ x,y \in \mS\]
gives the same value of metric mean dimension as Definition \ref{df:mmdim_subshift}:
\begin{prop}\label{prop:canonical} For a subshift $\mS\subset[0,1]^{\Z}$ it holds
	\[
	\overline{\mdim}_{M}(\mS,\sigma,\tau)=\limsup_{\varepsilon\to0}\lim_{n\rightarrow\infty}\frac{\log\#(\pi_{n}(\mS),||\cdot||_{\infty},\varepsilon)}{n\log\frac{1}{\varepsilon}}.
	\]
\end{prop}

For the proof of Proposition \ref{prop:canonical} we will use the following lemma:
\begin{lem}
	\label{lem: projection inequality} Let $\mS\subset[0,1]^{\Z}$
	be a subshift. Fix $\eps>0$ and $m\in\N$ such
	that $2^{-m+2}<\eps$. Then for $n\in\N$ and $A \subset \mS$
	the inequality
	\[
	\#(A,\tau_{n},8\eps)\leq\#(\pi_{-(m-1)}^{n+m}(A),\| \cdot \|_{\infty},\eps)
	\]
	holds ($\tau_n$ denotes here the dynamical metric $\tau_{n}(x,y)=\max \limits_{0\leq k<n}\tau(\sigma^{k}x,\sigma^{k}y)$).
\end{lem}

\begin{proof}
	Let $E\subset\pi_{-(m-1)}^{n+m}(A)$ be an $\eps$-net (recall Definition \ref{def:hashtag}) in the metric
	$\| \cdot \|_{\infty}$ on $[0,1]^{n+2m}$ with $\#E=\#(\pi_{-(m-1)}^{n+m}(A),\| \cdot \|_{\infty},\eps)$.
	Take $D\subset A$ consisting of representatives of sets $(\pi_{-(m-1)}^{n+m})^{-1}(\{x\}),\ x\in E$.
	Then $D$ is a $4\eps$-net in $A$ in the metric $\tau_{n}$.
	Indeed, for $y\in A$, there exists $x\in D$ such that $\|\pi_{-(m-1)}^{n+m}(y) -\pi_{-(m-1)}^{n+m}(x)\|_{\infty}<\eps$,
	hence for $0\leq j<n$ we have
	\[
	\tau(\sigma^{j}y,\sigma^{j}x)\leq2^{-m+2}+\sum\limits_{|k|<m}\frac{1}{2^{k}}|y_{k+j}-x_{k+j}| < 2^{-m+2}+ \eps\sum\limits_{|k|<m}\frac{1}{2^{k}}<4\eps.
	\]
	Taking a cover of $A$ by $4\eps$-balls with centers in $D$, we obtain the result.
\end{proof}

\begin{proof}(of Proposition \ref{prop:canonical})
	Observe first that for $x, y \in [0,1]^{\Z}$, the inequality $\tau_n(x,y) < \varepsilon$ implies $\|\pi_n(x) - \pi_n(y)\|_{\infty} < \varepsilon$ and hence
	\[ \# (\pi_n(\mS), \|\cdot\|_{\infty}, \varepsilon) \leq \#(\mS, \tau_n, \varepsilon). \]
	This gives
	\[ \limsup_{\varepsilon\to0}\lim_{n\rightarrow\infty}\frac{\log\#(\pi_{n}(\mS),||\cdot||_{\infty},\varepsilon)}{n\log\frac{1}{\varepsilon}} \leq \limsup_{\varepsilon\to0}\lim_{n\rightarrow\infty}\frac{\log\#(\mS,\tau_n,\varepsilon)}{n\log\frac{1}{\varepsilon}} =  \ummdim(\mS, \sigma, \tau).\]
	On the other hand, using Lemma \ref{lem: projection inequality} with $m = \lceil \log \frac{1}{\varepsilon} \rceil + 2$ and submultiplicativity of the function $n\mapsto\#(\pi_{n}(\mS),\|\cdot\|_{\infty},\eps)$ we obtain
	\[\#(\mS, \tau_n, 8\varepsilon) \leq \#(\pi_{-(m-1)}^{n+m}(\mS), \|\cdot\|_{\infty}, \varepsilon) \leq \#(\pi_n(\mS), \|\cdot\|_{\infty}, \varepsilon) \#(\pi_m(\mS), \|\cdot\|_{\infty}, \varepsilon)^2.  \]
	This yields
	\[\lim_{n\rightarrow\infty}\frac{\log\#(S,\tau_n,8\varepsilon)}{n}  \leq \lim \limits_{n \to \infty} \frac{ \log \#(\pi_{n}(\mS), \|\cdot\|_{\infty}, \varepsilon) + 2 \log \#(\pi_{m}(\mS), \|\cdot\|_{\infty}, \varepsilon)}{n} =\]
\[= \lim \limits_{n \to \infty} \frac{ \log \#(\pi_{n}(\mS), \|\cdot\|_{\infty}, \varepsilon)}{n}. \]
	Dividing both sides by $\log \frac{1}{8\varepsilon}$ and taking $\limsup \limits_{\varepsilon \to 0}$ ends the proof.
\end{proof}

\section{Proof of Proposition \ref{prop:mdim leq mbdim}}\label{app:proof_pro_mean_ineq}
We shall prove
\[	\overline{\mdim}_{M}(\mS)\leq\mbdim(\mS)\]
for a subshift $\mS \subset [0,1]^\Z$. Fix $\eta>0$. Take
$N\in\N$ with $\frac{1}{N}\overline{\dim}_{B}(\pi_{N}(\mS))\leq\mbdim(\mS)+\eta$.
Choose $\eps_{0}>0$ such that \[\#(\pi_{N}(\mS),\|\cdot\|_{\infty},\eps) \leq \eps^{-\overline{\dim}_{B}(\pi_{N}(\mS)) - \eta} \text{ for } 0<\eps<\eps_{0}.\]
Fix $\eps<\eps_{0}$. By the submultiplicativity of the function $n\mapsto\#(\pi_{n}(\mS),\|\cdot\|_{\infty},\eps)$, for $k \in \N$ we have
\[ \frac{\log \# (\pi_{kN}(\mS), \| \cdot\|_{\infty}, \eps)}{kN} \leq \frac{\log \#(\pi_{N}(\mS),\|\cdot\|_{\infty},\eps)}{N} \leq \frac{(\dim_B(\pi_{N}(\mS)) + \eta)\log \frac{1}{\eps}}{N} \leq\]
\[ \leq (\mbdim(\mS)+\eta(1 + \frac{1}{N}))\log \frac{1}{\eps}. \]
Therefore
\[ \ummdim(S) = \limsup \limits_{\eps \to 0} \lim \limits_{n \to \infty} \frac{\log \#(\pi_n(\mS), \| \cdot \|_{\infty}, \eps)}{n \log \frac{1}{\eps}} \leq \mbdim(\mS)+2\eta. \]
As $\eta>0$ was arbitrary, the proof is finished.

\section{Variational principle - details}\label{app:var_prin}
\numberwithin{thm}{subsection}

The goal of this section is to prove Theorem \ref{thm:var_prin_R_tilde}. It can be easily deduced from the variational principle for metric mean dimension in \cite{lindenstrauss_tsukamoto2017rate}, which is valid for dynamical systems more general than subshifts, once the corresponding rate-distortion functions are compared (see Proposition \ref{prop:R_J_comparison} below). Let us begin by formulating the original result of \cite{lindenstrauss_tsukamoto2017rate}.

\subsection{Variational principle for general dynamical systems}

\begin{defn}\label{def:tame_growth}(\cite[Condition II.3]{lindenstrauss_tsukamoto2017rate}) Let $(\mX, \rho)$ be a compact metric space. It is said to have the \textbf{tame growth of covering numbers} if for every $\delta>0$ we have
	\[ \lim \limits_{\eps \to 0} \eps^{\delta} \log \# (\mX, \rho, \eps) = 0. \]
\end{defn}
\begin{rem}\label{rem:tame_growth}
	Lindenstrauss and Tsukamoto observed that $([0,1]^\Z, \tau)$ has the tame growth of covering numbers, since $\log \#([0,1]^\Z, \tau, \eps) = O(|\log \eps |^2)$.
\end{rem}
In \cite{lindenstrauss_tsukamoto2017rate} the following definition
is made:

\begin{defn}\label{def:rate_dist_LT}
	Let $(\mX, T, \rho)$ be a dynamical system and $\mu\in \mP_{T}(\mathcal{X})$. For $ p \in [1, \infty)$ and $n\in\mathbb{{N}}$ define $\tilde{R}_{\mu,p}(\varepsilon,n)$ as the infimum
	of
	\[\frac{I(X;Y)}{n},\]
	where $X$ and $Y=(Y_{0},\dots,Y_{n-1})$ are random variables defined
	on some probability space $(\Omega,\mathbb{P})$ such that
	
	\begin{itemize}
		\item $X$ takes values in $\mathcal{X}$, and its law is given by $\mu$.
		\item Each $Y_{k}$ takes values in $\mathcal{X}$, and $Y$ approximates
		the process $(X,TX,\dots,T^{n-1}X)$ in the sense that
		\begin{equation}
		\mathbb{E}\left(\frac{1}{n}\sum_{k=0}^{n-1}\rho(T^{k}X,Y_{k})^p\right)<\varepsilon^p.\label{eq: distortion condition-1}
		\end{equation}
	\end{itemize}
	Similarly as in the Definition \ref{def:rate_distortion_function}, we can make use of the subadditivity of the sequence $n \mapsto n \tilde{R}_{\mu,p}(\eps, n)$ (which follows as in \cite[Theorem 9.6.1]{Gallager68}), to make the following definition:
	\[\tilde{R}_{\mu,p}(\varepsilon)=\lim_{n \to \infty} \tilde{R}_{\mu,p}(\varepsilon,n) = \inf_{n \in \N}\tilde{R}_{\mu,p}(\varepsilon,n).\]
	\begin{rem}\label{rem:rd_function_finite}
		As pointed out in \cite[Remark IV.3]{lindenstrauss_tsukamoto2017rate}, in Definition \ref{def:rate_dist_LT} it is enough to consider random vectors $Y$ taking finitely many values. As $I(X;Y) \leq H(Y) \leq \infty$ for such $Y$, we obtain also $\tilde{R}_{\mu, p}(\eps) < \infty$ for every $\eps>0$.
	\end{rem}
	
\end{defn}
In this setting, the variational principle for metric mean dimension states the following:

\begin{thm}\label{thm:lin_tsu_var_prin}(\cite[Corollary III.6]{lindenstrauss_tsukamoto2017rate}) Let $(\mX, \rho)$ be a compact metric space having the tame growth of covering numbers and let $T : \mX \to \mX$ be a homeomorphism. Then, for any $p \in [1, \infty)$
	\[ \ummdim(\mX, T, \rho) = \limsup \limits_{\eps \to 0} \frac{\sup_{\mu \in \mP_{T}(\mX)} \tilde{R}_{\mu, p}(\eps)}{\log \frac{1}{\eps}}\]
and
	\[ \lmmdim(\mX, T, \rho) = \liminf \limits_{\eps \to 0} \frac{\sup_{\mu \in \mP_{T}(\mX)} \tilde{R}_{\mu, p}(\eps)}{\log \frac{1}{\eps}}.\]
\end{thm}

Velozo and Velozo \cite{velozo2017rate} provided an alternative formulation in terms of Katok entropy. For the extension of the Theorem \ref{thm:var_prin_R_tilde} to actions of countable discrete amenable groups see \cite{CDZ17}.

\subsection{Proof of Theorem \ref{thm:var_prin_R_tilde}}

The difference between our formulation of Theorem \ref{thm:var_prin_R_tilde} and the original result of Lindenstrauss and Tsukamoto (Theorem \ref{thm:lin_tsu_var_prin}) is in the use of two different definitions of rate-distortion functions: Theorem \ref{thm:var_prin_R_tilde} uses $R_{\mu,p}$ as defined in Definition \ref{def:rate_distortion_function}, while Theorem \ref{thm:lin_tsu_var_prin} uses $\tilde{R}_{\mu,p}$ as defined in Definition \ref{def:rate_dist_LT}. The former one is more convenient to work with in the case of subshifts $\mS \subset A^\Z$, yet it is formally different from the latter one: $R_{\mu,p}$ is defined in terms of the metric $d$ on the alphabet $A$, while $\tilde{R}_{\mu,p}$ is defined in terms of the product metric
\begin{equation}\label{eq:product_metric2}\rho(x,y) = \sum \limits_{i \in \Z} \frac{d(x_i, y_i)}{2^{|i|}}
\end{equation}
on the space $\mS \subset A^\Z$. Therefore deducing Theorem \ref{thm:var_prin_R_tilde} from Theorem \ref{thm:lin_tsu_var_prin} requires a technical lemma comparing the two rate-distortion functions. Indeed, Theorem \ref{thm:var_prin_R_tilde} follows Theorem \ref{thm:lin_tsu_var_prin}, Proposition \ref{prop:R_J_comparison} and Remark \ref{rem:tame_growth} below.

The following proposition shows that for a subshift $\mS \subset [0,1]^\Z$ considered as a dynamical system $(\mS, \sigma, \tau)$, the above definition of the rate-distortion function is comparable with the one introduced in Definition \ref{def:rate_distortion_function}.

\begin{prop}\label{prop:R_J_comparison} Let $(A,d)$ be a compact metric space, $\mS \subset A^\Z$ a subshift and $\mu \in \mP_{\sigma}(\mS)$. Then, for any $p \in [1, \infty)$
	\[\tilde{R}_{\mu, p}(14\eps)\leq R_{\mu, p}(\varepsilon)\leq \tilde{R}_{\mu, p}(\varepsilon),\]
where $R_{\mu, p}$ is defined as in Definition \ref{def:rate_distortion_function}, while $\tilde{R}_{\mu, p}$ is defined as in Definition \ref{def:rate_dist_LT} with metric $\rho$ given by (\ref{eq:product_metric2}).
\end{prop}

\begin{proof}
	We will first show $R_{\mu,p}(\varepsilon)\leq \tilde{R}_{\mu, p}(\varepsilon)$.
	Fix $\delta>0$ and let $n \in \N$ be such that, $\tilde{R}_{\mu,p}(\varepsilon,n)<\tilde{R}_{\mu, p}(\varepsilon)+\delta$.
	In particular we may find random variables $X$ and $Y=(Y_{0},\dots,Y_{n-1})$ taking values in $\mS$ and $\mS^n$ respectively, obeying (\ref{eq: distortion condition-1}) and satisfying $\frac{I(X;Y)}{n}<\tilde{R}_{\mu, p}(\varepsilon)+\delta$.
	Define $\tilde{X}_k = \pi_0 (\sigma^k X)$ and $\tilde{Y}_{k}=\pi_0(Y_{k})$ for $k=0, \ldots, n-1$. Note that
	\[d(\tilde{Y}_{k},\tilde{X}_{k})\leq\rho(\sigma^{k}X,Y_{k}),\]
	hence $\tilde{Y}$ obeys (\ref{eq: distortion condition}). Thus
	$R_{\mu,p}(\varepsilon) \leq R_{\mu,p}(n,\varepsilon)\leq\frac{I((\tilde{X}_0, \ldots, \tilde{X}_{n-1});(\tilde{Y}_0, \ldots, \tilde{Y}_{n-1})}{n}\leq\frac{I(X;Y)}{n}<\tilde{R}_{\mu, p}(\varepsilon)+\delta$.
	We have used here the data-processing lemma \cite[Lemma 2.2]{lindenstrauss_tsukamoto2017rate}.
	
	We will show now $\tilde{R}_{\mu,p}(14\varepsilon)\leq R_{\mu,p}(\varepsilon)$.
	Fix $\delta>0$ and let $m,n \in \N$ be such that $2^{p(-m+2)-1}\diam(A)^p<\eps^p$, $(3\diam(A))^p\frac{2m}{n} < \eps^p,\ n > 2m$ and $R_{\mu,p}(\varepsilon,n)<R_{\mu,p}(\varepsilon)+\delta$.
	In particular for such $n$ we may find random variables
	$\tilde{X}=(\tilde{X}_{0},\dots,\tilde{X}_{n-1})$ and $\tilde{Y}=(\tilde{Y}_{0},\dots,\tilde{Y}_{n-1})$ taking values in $A^n$,
	obeying (\ref{eq: distortion condition}) and $\frac{I(\tilde{X};\tilde{Y})}{n}<R_{\mu,p}(\varepsilon)+\delta$.
	Let us denote by $d_p$ the metric on $A^n$ given by
	\begin{equation}\label{eq:dp_def} d_p(x,y) = \bigg(\frac{1}{n}\sum\limits_{i=0}^{n-1}d(x_i, y_i)^p\bigg)^{\frac{1}{p}}. \end{equation}
	Take any Borel map $S : A^n \to \mathcal{S},\ S= (S_i)_{i \in \Z}$ such that for $x \in A^n$ and $y = (y_i)_{i \in \Z} \in \mS$ the following holds:
	\begin{equation}\label{e:selector} d_p(x, \pi_n(S(x))) \leq d_p(x, \pi_n(y)).
	\end{equation}
	Formally, such $S$ can be constructed as follows. Define $t :A^n \to \R$ as
	\[t(x)=\min  \{d_p(x, \pi_n(y)) : y \in \mS\}\]
	($t$ is well defined as $(A,d)$ is compact). Let $C(\mS)$ be the space of closed non-empty subsets of $S$
	equipped with the Vietoris topology. Recall that its Borel space $B(C(\mS))$
	is the Effros Borel space of $\mS$, i.e. the $\sigma$-algebra generated by the sets of the form $\{ F \in C(\mS) : F \cap U \neq \emptyset \}$ for open $U \subset \mS$ (\cite[12.7]{K95}). Define $M : A^n \to C(\mS)$ as
	\[M(x) =\big\{y\in \mS : d_p(x, \pi_n(y)) = t(x)\big\}.\]
	
	Note that $M:[0,1]^{n}\rightarrow C(\mS)$ is Borel as $M^{-1}(\{F\in C(\mS)|\,F\cap U\neq\emptyset)\})$
	is open in $[0,1]^{n}$ for all $U$ open in $\mS$. By
	the Kuratowski-Ryll-Nardzewski selection theorem (\cite[12.13]{K95})
	there is a Borel selector $s:C(\mS)\rightarrow \mS$.
	Now $S:A^n\rightarrow \mS$ defined as $S(x)=s(M(x))$ satisfies (\ref{e:selector}).
	
	Since the distribution of $\tilde{X}$ is $(\pi_n)_* \mu$, there exist random variables $X = (X_k)_{k \in \Z}$ taking values in $\mS$ with distribution $\mu$ and $Z = (Z_0, \ldots, Z_{n-1})$ taking values in $A^n$ such that the conditional distributions satisfy $\mathbb{P}(Z | \pi_n(X) = x) = \mathbb{P}(\tilde{Y} | \tilde{X} = x)$. Define random variables $Y_0, \ldots, Y_{n-1}$ taking values in $\mS$ by $Y_k = \sigma^k \circ S \circ Z$. Using (\ref{e:selector}) we obtain
	\begin{equation}\label{eq:selector_comparision}
	\begin{gathered}	
	\mathbb{E} \big( d_p(\pi_n(X) , \pi_n(Y_0))\big)^p = \mathbb{E} \big( d_p(\pi_n(X) , \pi_n(S \circ Z) ) \big)^p \leq \\
 \leq  \mathbb{E} \big(d_p(\pi_n(X) , Z) + d_p(Z ,\pi_n(S \circ Z))\big)^p \\ \leq 2^p\mathbb{E}\big( d_p(\pi_n(X), Z)\big)^p = 2^p\mathbb{E}\big( d_p(\tilde{X} , \tilde{Y})\big)^p.
	\end{gathered}
	\end{equation}
	
We have the following series of inequalities (see below for explanations)

\begin{align}
\label{eq:rate_comp_1}\frac{1}{n}\sum_{k=0}^{n-1}\rho(\sigma^{k} X,Y_{k})^p & \leq(3\diam(A))^p\frac{2m}{n}+\frac{1}{n}\sum_{k=m}^{n-m-1}\rho(\sigma^{k}X,\sigma^{k} (S \circ Z))^p \\
\label{eq:rate_comp_2} & \leq\eps^p+\frac{1}{n}\sum_{k=m}^{n-m-1}\big(\sum_{i=-m}^{m}\frac{1}{2^{|i|}}d(X_{k+i},S_{k+i} \circ Z)+2^{-m+1}\diam(A)\big)^p \\
\label{eq:rate_comp_3} & \leq \eps^p+\frac{2^{p-1}}{n}\sum_{k=m}^{n-m-1}\Big(\big(\sum_{i=-m}^{m}\frac{1}{2^{|i|}}d(X_{k+i},S_{k+i} \circ Z)\big)^p +2^{p(-m+1)}\diam(A)^p\Big) \\
\label{eq:rate_comp_4} & \leq \eps^p + 2^{p(-m+2)-1}\diam(A)^p + \frac{3^{p-1}2^{p-1}}{n} \sum_{k=m}^{n-m-1} \sum_{i=-m}^{m}\frac{1}{2^{|i|}}d(X_{k+i},S_{k+i} \circ Z)^p \\
\label{eq:rate_comp_5} & \leq 2\eps^p + \frac{3^{p}2^{p-1}}{n} \sum_{k=0}^{n-1} d(X_{k},S_{k} \circ Z)^p \\
\label{eq:rate_comp_6} & = 2\eps^p + 3^{p}2^{p-1}\big( d_p(\pi_n(X) , \pi_n(Y_0))\big)^p.
\end{align}
Inequality (\ref{eq:rate_comp_1}) follows from equality $\diam(A^\Z) = 3 \diam(A)$, while (\ref{eq:rate_comp_2})  follows from (\ref{eq:product_metric2}). Inequality (\ref{eq:rate_comp_3}) is obtained by applying inequality $(a+b)^p \leq 2^{p-1}(a^p + b^p)$ for $a,b \geq 0,\ p \in [1, \infty)$ (which follows from Jensen's inequality \cite[Theorem 3.3]{R87}). (\ref{eq:rate_comp_4}) follows from Jensen's inequality (recall that $\sum \limits_{i = -\infty}^{\infty} \frac{1}{2^{|i|}}  = 3$), while (\ref{eq:rate_comp_5}) follows from (\ref{eq:product_metric2}) and (\ref{e:selector}). Equality (\ref{eq:rate_comp_6}) follows from the definition of $d_p$ (see (\ref{eq:dp_def})). Taking the expected value and estimating further we obtain
\begin{align}
\nonumber \mathbb{E} \Big( \frac{1}{n}\sum_{k=0}^{n-1}\rho(\sigma^{k} X,Y_{k})^p \Big) &\leq 2\eps^p + 3^{p}2^{p-1} \mathbb{E} \big( d_p(\pi_n(X) , \pi_n(Y_0))\big)^p \\
\label{eq:ex_rate_comp_1}& \leq 2\eps^p + 3^{p}2^{2p-1}\mathbb{E} \big( d_p(\tilde{X} , \tilde{Y})\big)^p \\
\label{eq:ex_rate_comp_2}& \leq (2 + 3^{p}2^{2p-1})\eps^p \\
\nonumber & \leq (2^p + 3^p2^{2p})\eps^p \leq  (14\eps)^p.
\end{align}
Inequality (\ref{eq:ex_rate_comp_1}) follows from (\ref{eq:selector_comparision}). For (\ref{eq:ex_rate_comp_2}) recall that $\tilde{X}, \tilde{Y}$ obey (\ref{eq: distortion condition}). We conclude that $Y = (Y_0, \ldots , Y_{n-1})$ obeys (\ref{eq: distortion condition-1}) with $14\eps$. Thus $\tilde{R}_{\mu,p}(14\eps) \leq \tilde{R}_{\mu,p}(14\eps,n)\leq\frac{I(X;Y)}{n}\leq\frac{I(\tilde{X};\tilde{Y})}{n}<R_{\mu,p}(\varepsilon)+\delta$. We have used here the data-processing lemma \cite[Lemma 2.2]{lindenstrauss_tsukamoto2017rate}.
\end{proof}

\begin{rem}
	Theorem \ref{thm:var_prin_R_tilde} holds true for subshifts $\mS \subset A^\Z$ for general compact metric space $(A,d)$, as long as $\mS$ has the tame growth of covering numbers (see Definition \ref{def:tame_growth}). This follows from the fact that Proposition \ref{prop:R_J_comparison} is true in such generality.
\end{rem}
\numberwithin{thm}{section}

\section{"Almost subadditivity" lemma}

The following lemma shows that, in some cases, the upper limit with respect to $n$ in the definition of the compression rate can be replaced by an infimum. The proof is based on the fact that, for suitable regularity classes $\mC$ and $\mD$, the sequence $n \mapsto n \r_{\mC - \mD}(\mS, 0, n)$ is "almost subadditive".

\begin{lem}\label{l:R0_subadd}
	Let $\mathcal{C}$ be a regularity class such that
\begin{itemize} \item $\mathcal{C}$ is closed under taking (tensor) products, i.e. if $f_1 : [0,1]^{n_1} \to [0,1]^{k_1}$, $f_2 : [0,1]^{n_2} \to [0,1]^{k_2}$ belong to $\mathcal{C}$ then so does $f : [0,1]^{n_1 + n_2} \to [0,1]^{k_1 + k_2}$ given by $f(x,y) = (f_1(x), f_2(y))$ for $x \in [0,1]^{n_1}, y \in [0,1]^{n_2}$,
\item the identity $id : [0,1]^n \to [0,1]^n$ belongs to $\mathcal{C}$ for every $n \in \N$.
\end{itemize}
	Then, for every $p \in [1, \infty], L\geq 1,\ \alpha \in (0,1]$ and $\eps > 0$, the following holds
\begin{equation}\label{eq:subadd_eps} \r_{\mathcal{C}-\mathcal{{H}}_{L+\eps,\alpha}^p}(\mS,0) \leq \inf \limits_{n \in \N} \r_{\mathcal{C}-\mathcal{{H}}_{L,\alpha}^p}(\mS,0,n).
\end{equation}
Consequently,
\begin{equation}\label{eq:subadd_inf} \inf \limits_{L>0} \r_{\mathcal{C}-\mathcal{{H}}_{L,\alpha}^p}(\mS,0) = \inf \limits_{L>0}\ \inf \limits_{n \in \N} \r_{\mathcal{C}-\mathcal{{H}}_{L,\alpha}^p}(\mS,0,n).
\end{equation}
In particular, (\ref{eq:subadd_eps}) and (\ref{eq:subadd_inf}) are true for $\mC \in \{ \mB, \LIN\}$.
\end{lem}

\begin{proof}
Let us begin by proving (\ref{eq:subadd_eps}). Fix $\delta > 0$ and choose $n_0 \in \N$ such that
\begin{equation}\label{eq:comp_rate_inf_n0}
\r_{\mathcal{C}-\mathcal{{H}}_{L,\alpha}^p}(\mS,0,n_0) \leq \inf \limits_{n \in \N} \r_{\mathcal{C}-\mathcal{{H}}_{L,\alpha}^p}(\mS,0,n) + \delta.
\end{equation}
Let $\r_{\mathcal{C}-\mathcal{{H}}_{L,\alpha}^p}(\mS,0,n_0) = \frac{k_0}{n_0}$, i.e. there exist functions $f_0 : [0,1]^{n_0} \to [0,1]^{k_0}$ and $g_0 : [0,1]^{k_0} \to [0,1]^{n_0}$ such that $f_0 \in \mC,\ g_0 \in \mH^p_{L, \alpha}$ and

\begin{equation}\label{eq:comp_rate_mu_f_0} \mu(\{ x \in \mS : g_0 \circ f_0 (x|_0^{n_0-1}) \neq x|_{0}^{n_0-1}\}) = 0 \end{equation}
for every $\mu \in \mP_{\sigma}(\mS)$. As $\id \in \mC$, we can assume that $k_0 \leq n_0$.

Fix $n \in \N$ and write it (uniquely) as $n = \ell n_0 + m$, where $\ell, m \in \N_0$ and $0 \leq m < n_0$. Define $f : [0,1]^n \to [0,1]^{\ell k_0 + m}$ as the $\ell$-fold product of $f_0$ followed by the identity on the remaining $m$ coordinates. More precisely, we set
\begin{align*} f(x_0, x_1, &\ldots , x_{n-1}) \\
& := (f_0(x_0, \ldots, x_{n_0-1}), f_0(x_{n_0}, \ldots, x_{2n_0 - 1}), \ldots, f_0(x_{(\ell-1)n_0}, \ldots, x_{\ell n_0 - 1}), x_{\ell n_0}, \ldots, x_{\ell n_0 + m-1}).
\end{align*}
By the assumptions on $\mC$, $f$ belongs to $\mC$. Similarly, we define $g : [0,1]^{\ell k_0 + m} \to [0,1]^n$ as the $\ell$-fold concatenation of $g_0$ followed by the identity on the remaining $m$ coordinates. More precisely, we set
\begin{align*}  g(x_0, x_1, &\ldots , x_{\ell k_0 + m}) \\ &:= (g_0(x_0, \ldots, x_{k_0-1}), g_0(x_{k_0}, \ldots, x_{2k_0 - 1}), \ldots, g_0(x_{(\ell-1)k_0}, \ldots, x_{\ell k_0 - 1}), x_{\ell k_0}, \ldots, x_{\ell k_0 + m-1}).
\end{align*}
The shift-invariance of $\mu$ together with (\ref{eq:comp_rate_mu_f_0}) and definitions of $f$ and $g$ give
\[ \mu(\{ x \in \mS : g \circ f (x|_0^{n-1}) \neq x|_{0}^{n-1}\}) = \mu(\{ x \in \mS : \underset{ 0 \leq j < \ell}{\exists}\ g_0\circ f_0 (x|_{jn_0}^{(j+1)n_0-1}) \neq x|_{jn_0}^{(j+1)n-1}\}) =  \]
\[ = \mu\Big( \bigcup \limits_{j=0}^{\ell-1} \sigma^{-jn_0}(\{ x \in \mS : g_0\circ f_0 (x|_{0}^{n_0-1}) \neq x|_{0}^{n_0-1}\})\Big) = 0. \]
Let us show now that for large $n$, the map $g$ belongs to $\mH^p_{L+\eps, \alpha}$, i.e. it is $(L+\eps,\alpha)$-H\"{o}lder in the norm $\| \cdot \|_p$. Let us consider the case $p \in [1,\infty)$, as the case $p=\infty$ is straightforward. Fix $x,y \in [0,1]^{\ell k_0 + m}$. We have the following series of inequalities (see below for explanations)

\begin{align}
\nonumber  \|g(x) - g(y)\|_p^p &=  \\
&\begin{gathered}\label{eq:lp_norm_holder_1} =  \frac{n_0}{\ell n_0+m} \sum \limits_{i=0}^{\ell-1} \|g_0(x|_{ik_0}^{(i+1)k_0 -1}) - g_0(y|_{ik_0}^{(i+1)k_0 -1}) \|_p^p
+ \frac{1}{\ell n_0+m}\sum \limits_{i=0}^{m-1} |x_{\ell k_0 + i} - y_{\ell k_0 +i}|^p
\end{gathered} \\
&\begin{gathered}\label{eq:lp_norm_holder_2} \leq L^p \Big(\frac{n_0}{\ell n_0 + m} \sum \limits_{i=0}^{\ell-1} \|x|_{ik_0}^{(i+1)k_0 -1} - y|_{ik_0}^{(i+1)k_0 - 1} \|_p^{p\alpha}
+ \frac{1}{\ell n_0+m} \sum \limits_{i=0}^{m-1} |x_{\ell k_0 + i} - y_{\ell k_0 +i}|^{p\alpha} \Big)
\end{gathered} \\
& \begin{gathered}\label{eq:lp_norm_holder_3}  \leq  L^p \Big(\frac{n_0}{\ell n_0 + m} \sum \limits_{i=0}^{\ell-1} \|x|_{ik_0}^{(i+1)k_0 -1} - y|_{ik_0}^{(i+1)k_0 - 1} \|_p^{p}
+ \frac{1}{\ell n_0+m} \sum \limits_{i=0}^{m-1} |x_{\ell k_0 + i} - y_{\ell k_0 +i}|^{p} \Big)^\alpha
\end{gathered}\\
& \begin{gathered}\label{eq:lp_norm_holder_4}  \leq  L^p \Big(\frac{n_0 (\ell k_0 + m) }{k_0 (\ell n_0 + m)}\Big)^{\alpha} \Big(\frac{k_0}{\ell k_0 + m} \sum \limits_{i=0}^{\ell-1} \|x|_{ik_0}^{(i+1)k_0 -1} - y|_{ik_0}^{(i+1)k_0 - 1} \|_p^{p} \\
\hfill + \frac{1}{\ell k_0+m} \sum \limits_{i=0}^{m-1} |x_{\ell k_0 + i} - y_{\ell k_0 +i}|^{p} \Big)^\alpha
\end{gathered}\\
&\label{eq:lp_norm_holder_5}  = L^p \Big(\frac{n_0 (\ell k_0 + m) }{k_0 (\ell n_0 + m)}\Big)^{\alpha} \| x - y\|_p^{p\alpha} \\
\nonumber &\leq L^p \Big(\frac{n_0 (\ell k_0 + n_0) }{\ell k_0n_0}\Big)^{\alpha} \| x - y\|_p^{p\alpha} \\
\nonumber & = L^p \Big(1 + \frac{n_0}{\ell k_0}\Big)^{\alpha}\| x - y\|_p^{p\alpha}.
\end{align}

Equality (\ref{eq:lp_norm_holder_1}) follows from the definition of the $p$-th norm (recall that it is normalized). Inequality (\ref{eq:lp_norm_holder_2}) follows from $g_0 \in \mH^p_{L,\alpha},\ L \geq 1,\ \alpha \in (0,1]$ and $x,y \in [0,1]^{\ell k_0 + m}$. Applying Jensen's inequality (see e.g. \cite[Theorem 3.3]{R87}) yields (\ref{eq:lp_norm_holder_3}). As $k_0 \leq n_0$, inequality (\ref{eq:lp_norm_holder_4}) follows. For (\ref{eq:lp_norm_holder_5}) we use the definition of the $p$-th norm once more.
If $n$ is large enough, then $L^p \Big(1 + \frac{n_0}{\ell k_0}\Big)^{\alpha} \leq (L + \eps)^p$ (since $\ell \to \infty$ as $n \to \infty$).
We have therefore found a compressor-decompressor pair consisting of $f \in \mC$ and $g \in \mH^p_{L + \eps,\alpha}$, hence we can estimate (for $n$ large enough)
\[ \r_{\mathcal{C}-\mathcal{{H}}_{L+\eps,\alpha}^p}(\mS,0,n) \leq \frac{\ell k_0 + m}{n} = \frac{\ell k_0 + m}{\ell n_0 + m} \leq \frac{k_0}{n_0} + \frac{n_0-1}{n} = \r_{\mathcal{C}-\mathcal{{H}}_{L,\alpha}^p}(\mS,0,n_0) + \frac{n_0-1}{n}.  \]
Taking $\limsup \limits_{n \to \infty}$ and applying (\ref{eq:comp_rate_inf_n0}), we obtain
\[ \limsup \limits_{n \to \infty} \r_{\mathcal{C}-\mathcal{{H}}_{L+\eps,\alpha}^p}(\mS,0,n) \leq \r_{\mathcal{C}-\mathcal{{H}}_{L,\alpha}^p}(\mS,0,n_0) \leq \inf \limits_{n \in \N}\ \r_{\mathcal{C}-\mathcal{{H}}_{L,\alpha}^p}(\mS,0,n) + \delta. \]
As $\delta>0$ was arbitrary, (\ref{eq:subadd_eps}) is proved.
For (\ref{eq:subadd_inf}) observe that the definition of $\r_{\mathcal{C}-\mathcal{{H}}_{L+\eps,\alpha}^p}(\mS,0)$ and (\ref{eq:subadd_eps}) imply that
\[ \inf \limits_{n \in \N} \r_{\mathcal{C}-\mathcal{{H}}_{L + \eps,\alpha}^p}(\mS,0,n) \leq \r_{\mathcal{C}-\mathcal{{H}}_{L+\eps,\alpha}^p}(\mS,0) \leq \inf \limits_{n \in \N} \r_{\mathcal{C}-\mathcal{{H}}_{L,\alpha}^p}(\mS,0,n).\]
Taking $ \inf \limits_{L>0}$ on both sides yields (\ref{eq:subadd_inf}), as the function $(0, \infty) \ni L \mapsto \inf \limits_{n \in \N} \r_{\mathcal{C}-\mathcal{{H}}_{L,\alpha}^p}(\mS,0,n)$ is non-increasing in $L$.
\end{proof}

\section{Examples - calculations}\label{app:examples}

\subsection{Example \ref{e:mmdim_0_mbdim_1}}\label{app:mmdim_0_mbdim_1}
Let
\[ \mS_m = \bigcup\limits_{n\in \Z}\sigma^{n}(A_{m}) \text{ and }\mS:=\bigcup\limits_{m\geq1} \mS_m,\]
where 
\[A_{m}:=\{ (\ldots, 0, 0)\} \times [0,\frac{1}{2^{m}}]^{m} \times \{ (0,0, \ldots) \}\subset[0,1]^{\Z},\]
with the cube $[0,\frac{1}{2^{m}}]^{m}$ located on coordinates $0, 1, \ldots, m-1$.
We shall prove
	\[ \mmdim(\mS) = 0\ \text{ and }\ \mbdim(\mS) = 1.\]
$\mS$ is clearly $\sigma$-invariant. Let $\vec{0} = (\ldots, 0, 0, 0, \ldots ) \in [0,1]^\Z$. Note that each $\mS_m$ is compact (in the metric $\tau$) with
	\[\diam(\mS_m) \leq \frac{2m}{2^m} \text{ and } \vec{0} \in \mS_m,\]
	hence $\mS$ is compact.
	Moreover, for every $x\in \mS$ it holds $\sigma^{n}x \to \vec{0}$
	as $n\to\infty$, hence $\mP_{\sigma}(\mS) = \{\delta_{\vec{0}}\}$. It follows from the Variational Principle (Theorem \ref{thm:var_prin_R_tilde}) that $\ummdim(\mS) = 0$. On
	the other hand, for every $m\geq1$ it holds
	\[
	[0,\frac{1}{2^{m}}]^{m}\subset\pi_{m}(\mS),
	\]
	hence $\udim(\pi_m(\mS)) \leq m$ and therefore $\mbdim(\mS)=1$.

\subsection{Example \ref{ex:two_symbols_shift}}\label{sub:two_symbols_shift}
The goal of this subsection is to prove  inequality (\ref{eq:two_symbols_shift_lower}). We will need the following lemma.

\begin{lem}\label{lem:binomial_sum_entropy}
Fix $\delta \in (0,\frac{1}{2}], n \in \N$ and $x \in \{0,1 \}^n$. Let $B(x,\delta)$ be the ball in the norm $\| \cdot \|_1$ on  $\{0,1\}^n$, i.e.
\[ B(x,\delta) = \{ y \in \{0,1 \}^n : \|x - y \|_1 < \delta\}. \]
Then $\#B(x,\delta) \leq 2^{nH(\delta)}$, where $H(\delta) = -\delta \log \delta - (1-\delta) \log (1 - \delta)$.
\end{lem}
\begin{proof}
Fix $y \in B(x, \delta)$. Then
\[ \frac{1}{n} \sum \limits_{i=0}^{n-1} |x_i - y_i| < \delta, \text{ where } x = (x_0, \ldots, x_{n-1}), y = (y_0, \ldots, y_{n-1}). \]
Therefore, the set $I_y = \{ 0 \leq i < n : x_i \neq y_i \}$ satisfies $I_y \leq \delta n$. Moreover, vector $y$ is uniquely determined by $I_y$, hence the assignment $B(x, \delta) \ni y \mapsto I_y \subset \{0, \ldots, n-1\}$ is injective. Consequently
\[ \#B(x, \delta) \leq \# \{ I \subset \{0, \ldots, n-1 \} : \#I \leq n\delta \} = \sum \limits_{0 \leq i \leq n\delta} {n\choose i} \leq 2^{nH(\delta)}, \]
where the last inequality follows from \cite[Lemma 3.6]{Gray11} (recall that we take logarithms in the base $2$).
\end{proof}

Let $\mS := \{0,1\}^{\Z}$. We will show that for any $ \alpha \in (0,1]$ and $L\geq1$ it holds that
\begin{equation}\label{eq:two_symbols_shift_lower1} \sup \limits_{\varepsilon > 0} \ \sup \limits_{\mu \in  \mP_{\sigma}(\mS)} \r^{L^1}_{\mathcal{B} - \mathcal{H}_{L, \alpha}^1}(\mu, \eps) \geq \frac{\alpha (1 - H(\frac{1}{4}))}{\log 8 L} > 0.
\end{equation}
As the quantity $\r^{L^1}_{\mathcal{B} - \mathcal{H}^1_{L, \alpha}}(\mu, \eps)$ is decreasing in $L$, this will be sufficient for proving (\ref{eq:two_symbols_shift_lower}) for all $L>0$.
We will find a single measure for which the lower bound in (\ref{eq:two_symbols_shift_lower1}) holds. Let $\mu = \bigotimes \limits_{\Z} (\frac{1}{2}\delta_0 + \frac{1}{2}\delta_1) \in \mP_{\sigma}({\mS})$. Fix $\varepsilon \in (0, \frac{1}{16})$. Fix $n \in \N$ and let $\r^{L^1}_{\mathcal{B} - \mathcal{H}_{L, \alpha}^1} (\mu, \varepsilon, n) = \frac{k}{n}$ for some $k \in \N$. There exists then a Borel map $f : [0,1]^n \to [0,1]^k$ and $(L, \alpha)$-H\"{o}lder map $g : [0,1]^k \to [0,1]^n$  (in norm $\| \cdot \|_1$) such that
\begin{equation}\label{eq:two_symbols_shift_L1_error}\int \limits_{[0,1]^n} \|x - g \circ f (x)\|_1 d (\pi_n)_* \mu(x) \leq \eps.
\end{equation}
	Clearly $(\pi_n)_* \mu = (\frac{1}{2}\delta_0 + \frac{1}{2}\delta_1)^{\otimes n}$. Let $A_n \subset \{ 0,1 \}^n$ be defined as $A_n = \{ x \in \{0,1\}^n : \|x - g \circ f (x) \|_1 \leq 2\eps \}$. By the Chebyshev inequality (see e.g. \cite[(5.30)]{BillingsleyPM}) and (\ref{eq:two_symbols_shift_L1_error}) we have
	\[\#\big( \{ 0,1 \}^n \setminus A_n) \leq \frac{2^n}{2\eps} \int \limits_{\{ 0,1 \}^n \setminus A_n} \|x - g \circ f (x)\|_1 d (\pi_n)_* \mu(x) \leq 2^{n-1}. \]
Therefore
\begin{equation}\label{eq:two_symbols_shift_A_sep}
\#A_n \geq 2^{n-1}.
\end{equation} Lemma \ref{lem:binomial_sum_entropy} (applied with $\delta = 8 \eps$) implies that there exists a subset $B_n \subset A_n$ satisfying
\begin{equation}\label{eq:two_symbols_shift_B_sep} \|x - y \|_1 \geq 8\eps \text{ for every } x,y \in B_n \text{ such that } x \neq y
\end{equation}
and
\begin{equation}\label{eq:two_symbols_shift_B_card} \#B_n \geq \frac{\#A_n}{2^{nH(8\eps)}}.\end{equation}
By the triangle inequality, (\ref{eq:two_symbols_shift_B_sep}) and the definition of $A_n$, for distinct $x,y \in B_n$ we have
\[ \| g \circ f (x) - g \circ f(y) \|_1 \geq \|x - y\|_1 - \|x - g \circ f (x)\|_1 - \|y - g \circ f (y)\|_1 \geq 4 \eps.\]
and consequently, as $g$ is $(L,\alpha)$-H\"{o}lder in $\| \cdot \|_1$,
\begin{equation}\label{eq:two_symbols_shift_f_sep}
4 \eps \leq \| g \circ f (x) - g \circ f(y) \|_1 \leq L \| f (x) - f(y) \|_1^{\alpha} \text{ for } x,y \in B_n,\ x \neq y.
\end{equation}
This implies that $f$ is injective on $B_n$, hence $\#f(B_n) = \# B_n$. Moreover, by (\ref{eq:two_symbols_shift_f_sep}), elements of the set $f(B_n)$ are $ \big(\frac{4\eps}{L}\big)^{\frac{1}{\alpha}}$-separated in the norm $\| \cdot\|_1$ on $[0,1]^k$, hence $\#f(B_n) \leq \big(\frac{4\eps}{L}\big)^{-\frac{k}{\alpha}}$ (as $f(B_n)$ is $\big(\frac{4\eps}{L}\big)^{\frac{1}{\alpha}}$-separated in the norm $\| \cdot\|_\infty$ as well, since $\|\cdot\|_{\infty} \geq \| \cdot \|_1$). Combining this with (\ref{eq:two_symbols_shift_A_sep}) and (\ref{eq:two_symbols_shift_B_card}) gives bounds
\[ 2^{n(1 - H(8\eps))-1} \leq \#B_n = \#f(B_n) \leq  \big(\frac{4\eps}{L}\big)^{-\frac{k}{\alpha}}. \]
Taking logarithms we obtain
\[ n(1 - H(8\eps))-1 \leq \frac{k}{\alpha} \log \frac{L}{4\eps},\]
therefore (recall that $L\geq 1$ and $\eps < \frac{1}{16}$, hence $\frac{L}{4\eps} > 1$ and $H(8\eps) < 1$)
\[  \r^{L^1}_{\mathcal{B} - \mathcal{H}_{L, \alpha}^1} (\mu, \varepsilon, n) = \frac{k}{n} \geq \frac{\alpha(1 - H(8\eps))}{ \log \frac{L}{4\eps}} - \frac{1}{n \log \frac{L}{4\eps}}. \]
This implies
\[ \r^{L^1}_{\mathcal{B} - \mathcal{H}_{L, \alpha}^1} (\mu, \varepsilon) = \limsup \limits_{n \to \infty}\ \r^{L^1}_{\mathcal{B} - \mathcal{H}_{L, \alpha}^1} (\mu, \varepsilon, n) \geq \frac{\alpha(1 - H(8\eps))}{ \log \frac{L}{4\eps}}. \]
We finally obtain
\[ \sup \limits_{\eps > 0}\ \r^{L^1}_{\mathcal{B} - \mathcal{H}_{L, \alpha}^1} (\mu, \varepsilon) \geq \r^{L^1}_{\mathcal{B} - \mathcal{H}_{L, \alpha}^1} (\mu, \frac{1}{32}) \geq  \frac{\alpha(1 - H(\frac{1}{4}))}{ \log 8L}. \]
Taking supremum over $\mu \in \mP_{\sigma}(\mS)$ gives (\ref{eq:two_symbols_shift_lower1}).

\subsection{Example \ref{ex:L_must_stay}}\label{sub:L_must_stay}
Let $A = \{0\} \cup \{\frac{1}{n} : n \in \N \}$ and $\mS = A^{\Z}$. Then, according to Example \ref{ex:general_full_shift},
\[\ummdim(\mS) = \mbdim(\mS) = \udim(A) = \frac{1}{2}.\]
We will prove
\[\sup \limits_{\eps > 0}\ \sup \limits_{\mu \in  \mP_{\sigma}(\mS)}\ \r_{\mathcal{B} - \mathcal{H}_{\alpha}}(\mu, \eps) = 0 \text{ for every }\alpha \in (0,1].\] To that end, fix $\varepsilon > 0,\ n \in \N$ and $\mu \in  \mP_{\sigma}(\mS)$. Since $(\pi_n)_* \mu$ is a discrete measure with countably many atoms, there exists a finite set $B \subset [0,1]^n$ with $(\pi_n)_* \mu (B) \geq 1 - \varepsilon$. Since $B$ is finite, there exists a linear map $F : \R^n \to \R$ such that $F([0,1]^n) \subset [0,1]$ and $F$ is injective on $B$ (e.g. projection onto a suitable line - there is only finitely many directions such that projection onto them will fail to be injective on a given finite set). Let $f : [0,1]^n \to [0,1]$ be given by $f = F|_{[0,1]^n}$. Let $g: f(B) \to [0,1]^n$ be its inverse on $f(B)$. Note that $g$ is $L$-Lipschitz with respect to the norm $\| \cdot \|_{\infty}$ for $L = \Big(\min \big\{ \|x - y\|_{\infty} : x, y \in f(B),\ x \neq y \big\}\Big)^{-1}$. The function $g$ can be extended to an $L$-Lipschitz map from $[0,1]$ to $[0,1]^n$ for some $L>0$ (see Remark \ref{rem:extension_theorems}). Now
	\[ \mu(\{ x \in \mS : g\circ f \circ \pi_n(x) = \pi_n(x) \}) \geq (\pi_n)_* \mu(B) \geq 1 - \varepsilon, \]
	hence $\r_{\mathrm{LIN} - \mathcal{L}}(\mu, n, \varepsilon) \leq \frac{1}{n}$ and $\r_{\mathrm{LIN} - \mathcal{L}}(\mu, \eps) = 0$. Consequently
	\[\sup \limits_{\eps > 0}\ \sup \limits_{\mu \in  \mP_{\sigma}(\mS)}\ \r_{\mathcal{B} - \mathcal{H}_{\alpha}}(\mu, \eps) \leq \sup \limits_{\eps > 0}\  \sup \limits_{\mu \in  \mP_{\sigma}(\mS)}\ \r_{\mathrm{LIN} - \mathcal{L}}(\mu, \eps) = 0. \]
	Note that in the above calculation we cannot guarantee a uniform bound on $L$ (with respect to $n$) as $\Big(\min \big\{ \|x - y\|_{\infty} : x, y \in f(B),\ x \neq y \big\}\Big)^{-1}$ can be arbitrary large for $n \to \infty$.

\subsection{Example \ref{ex:sparse_b-h_equality}}\label{sec:sparse_b-h_equality}
Let $\mS \subset[0,1]^\Z$ be the subshift from Example \ref{ex:sparse_subshift}, i.e.
\[ \mS = \{ x \in [0,1]^\Z : \underset{j \in \Z}{\forall}\ \|x|_{j}^{j+N-1} \|_0 \leq K\} \]
for $N,K \in \N,\ K \leq N$. We shall prove that for every $p \in [1,\infty], \alpha \in (0,1]$ it holds
\begin{equation}\label{eq:sparse_borel_rate2} \inf \limits_{L>0}\ \sup \limits_{\mu \in \mP_{\sigma}(\mS)} \sup \limits_{\eps>0}\ \r_{\mB - \mH^p_{L,\alpha}}(\mu, \eps) = \inf \limits_{L>0}\ \r_{\mB - \mH^p_{L,\alpha}}(\mS, 0) = \alpha \mmdim(\mS) = \frac{\alpha K}{N}.
\end{equation}
Theorem \ref{thm:main_holder_low} and the calculation from Example \ref{ex:sparse_subshift} give
\[ \frac{\alpha K}{N} = \alpha \mmdim(\mS) \leq  \inf \limits_{L>0}\ \sup \limits_{\mu \in \mP_{\sigma}(\mS)}\ \sup \limits_{\eps>0}\ \r_{\mB - \mH^p_{L,\alpha}}(\mu, \eps),\]
hence it suffices to prove
\begin{equation}\label{eq:sparse_unif_borel_rate} \inf \limits_{L>0}\ \r_{\mB - \mH^p_{L,\alpha}}(\mS, 0) \leq \frac{\alpha K}{N}.
\end{equation}
We will apply Lemma \ref{l:R0_subadd} and Proposition \ref{prop:holder_peano_cubes}. Fix $\eps>0$. Let $\ell \in \N$ be such that $\frac{3}{\ell N} \leq \eps$. Set $d:=\lceil \alpha \ell K \rceil + 1$. Then $\alpha < \frac{d}{\ell K}$, hence by Proposition \ref{prop:holder_peano_cubes} there exist $\phi : [0,1]^{\ell K} \to [0,1]^d,\ \psi:[0,1]^d \to [0,1]^{\ell K}$ such that $\phi \in \mB,\ \psi \in \mH_{\alpha}$ and $\psi \circ \phi(x) = x$ for every $ x \in [0,1]^{\ell K}$. Define the set
\[ \mA := \{ A \subset \{0, \ldots, \ell N - 1\} : \#A = \ell K \} \]
consisting of all subsets of $\{0, \ldots, \ell N-1 \}$ of cardinality $\ell K$. Let $C : \pi_{\ell N}(\mS) \to \mA$ be any Borel map such that $\supp(x) \subset C(x)$ holds for every $x \in \pi_{\ell N}(\mS)$. Such map exists as $\#\supp(x) \leq \ell K$ for every $x \in \pi_{\ell N}(\mS)$ (for vectors $x \in \pi_{\ell N}(\mS)$ with $\#\supp(x)<\ell K$, one can set $C(x)$ to be e.g. the union of $\supp(x)$ with $\ell K-\#\supp(x)$ first zero coordinates of $x$. This is clearly a Borel map). Moreover, let $s : \mA \to [\frac{1}{2},1]$ be a map such that $|s(A) - s(B)| \geq \frac{1}{2(\#\mA - 1)}$ for $A \neq B$ (i.e. it assigns to sets $A \in \mA$ points from $[\frac{1}{2}, 1]$ in a maximally separated fashion). We are now ready to define a compressor $f:[0,1]^{\ell N} \to [0,1]^{d+1}$ and a decompressor $g:[0,1]^{d+1} \to [0,1]^{\ell N}$ with $f \in \mB$ and $g \in \mH^p_{L,\alpha}$ for some $L$ large enough. The compressor $f$ will assign to $x\in \pi_{\ell N}(\mS)$ the image under $\phi$ of $x$ restricted to the set $C(x)$ (which contains the support of $x$), together with a signature $s(C(x))$. Knowing $s(C(x))$ (and hence $C(x)$ itself) will allow us to decode $x$ uniquely by applying $\psi$ and setting coordinates from the complement of $C(x)$ to zero. The fact that $\psi$ is $\alpha$-H\"{o}lder and $s$ separates sets from $\mA$ will allow us to conclude that the decompressor is H\"{o}lder as well. Specifically let us define $f:[0,1]^{\ell N} \to [0,1]^{d+1}$,
\[ f(x) = \begin{cases}
0, & x \notin \pi_{\ell N}(\mS)\\
(\phi(x|_{C(x)}), s(C(x)), & x \in \pi_{\ell N}(\mS)
\end{cases}. \]
Clearly $f$ is Borel. Define $g : f([0,1]^{\ell N}) \to [0,1]^{\ell N}$ as
\[ g(x_1, \ldots., x_d, x_{d+1}) = \begin{cases}
0, & x_{d+1} = 0\\
\psi(x_1, \ldots, x_d) \upharpoonleft s^{-1}(x_{d+1}), & x_{d+1} \neq 0
\end{cases}, \]
where $\psi(x_1, \ldots, x_d) \upharpoonleft s^{-1}(x_{d+1})$ denotes the vector in $[0,1]^{\ell N}$ obtained by putting the consecutive elements of the vector $\psi(x_1, \ldots, x_d)$ (which has length $\ell K$) in the consecutive coordinates from the set $s^{-1}(x_{d+1}) \in \mA$ (of cardinality $\ell K$ as well) and setting other coordinates to zero. Note that $g$ is well defined on $f([0,1]^{\ell N})$ and $g \circ f = \id$ on $\pi_{\ell N}(\mS)$. Let us show now that $g$ is $\alpha$-H\"{o}lder on $f([0,1]^{\ell N})$. As $\psi \in \mH_{\alpha}$, there exists $L>0$ such that $\psi$ is $(L,\alpha)$-H\"{o}lder in the norm $\| \cdot \|_p$. Note that $f([0,1]^{\ell N}) \subset [0,1]^d \times \{0, \frac{1}{2}, \frac{1}{2} + \frac{1}{2(\#\mA - 1)}, \ldots, 1 - \frac{1}{2(\#\mA - 1)}, 1 \}$. Fix $x = (x_1, \ldots, x_d, x_{d+1}),\ y = (y_1, \ldots, y_d, y_{d+1}) \in f([0,1]^{\ell N})$. Assume first that $x_{d+1} \neq y_{d+1}$. Then $|x_{d+1} - y_{d+1}| \geq \frac{1}{2(\#\mA - 1)}$, hence $\|x - y\|_p \geq M$, where $M = \frac{1}{(d+1)^\frac{1}{p}2(\#\mA - 1)}$ for $p \in (1,\infty)$ and $M :=  \frac{1}{2(\#\mA - 1)}$ for $p=\infty$. Therefore

\begin{equation}\label{eq:holder_d+1_neq}\|g(x) - g(y)\|_p \leq 1 \leq \frac{1}{M^{\alpha}}\|x - y\|_p^{\alpha} \text{ if } x_{d+1} \neq y_{d+1}
\end{equation}
Assume now that $x_{d+1} = y_{d+1}$. By $(L,\alpha)$-H\"{o}lder continuity of $\psi$ on $[0,1]^d$, we have

\begin{equation}\label{eq:holder_d+1_eq}
\begin{aligned}
\|g(x)-g(y)\|_p &\leq \| \psi(x_1, \ldots, x_d) - \psi(x_1, \ldots, x_d) \|_p \\ &\leq L \| (x_1,\ldots, x_d) - (y_1, \ldots, y_d) \|_p^{\alpha} \\
& \leq \tilde{L} \|x -y \|_p^{\alpha} \text{ if } x_{d+1} = y_{d+1},
\end{aligned}
\end{equation}
where $\tilde{L} := \frac{L(d+1)^{\frac{\alpha}{p}}}{d^{\frac{\alpha}{p}}}$ if $p \in (1,\infty)$ and $\tilde{L} := L$ if $p=\infty$. Setting $L' = \max \{ \frac{1}{M^\alpha},\tilde{L} \}$, we can conclude from (\ref{eq:holder_d+1_neq}) and (\ref{eq:holder_d+1_eq}) that $g$ is $(L', \alpha)$-H\"{o}lder on $f([0,1]^{\ell N})$. It can be extended to $(L'', \alpha)$-H\"{o}lder function $g:[0,1]^{d+1} \to [0,1]^{\ell N}$, where $L'' = (\ell N)^\frac{1}{p}L'$ if $p\in[1,\infty)$ and $L'' = L'$ if $p=\infty$ (see Remark \ref{rem:extension_theorems}). As noted above, $g \circ f = \id$ on $\pi_{\ell N}(\mS)$, hence $\mu(\{ x \in \mS : g \circ f (x|_0^{\ell N-1}) \neq x|_0^{\ell N-1} \}) = 0$ for every $\mu \in \mP_{\sigma}(\mS)$. We therefore obtain bound
\[ \r_{\mB - \mH^p_{L'', \alpha}}(\mS, 0, \ell N) \leq \frac{d+1}{\ell N} = \frac{\lceil \alpha \ell K \rceil +2}{\ell N} \leq \frac{\alpha K}{N} + \frac{3}{\ell N} \leq \frac{\alpha K}{N} + \eps.  \]
By Lemma \ref{l:R0_subadd} applied with $\mC = \mB$, we arrive at
\[ \inf_{L>0} \r_{\mB - \mH^p_{L, \alpha}}(\mS, 0) \leq  \r_{\mB - \mH^p_{L'', \alpha}}(\mS, 0, \ell N) \leq \frac{\alpha K}{N} + \eps. \]
As $\eps>0$ was arbitrary, (\ref{eq:sparse_unif_borel_rate}) is proved and consequently (\ref{eq:sparse_borel_rate2}) holds.

\subsection{Example \ref{ex:sparse_LIN-H}}\label{sub:sparse_LIN-H}

Let $\mS \subset[0,1]^\Z$ be the subshift from Example \ref{ex:sparse_subshift}, i.e.
\[ \mS = \{ x \in [0,1]^\Z : \underset{j \in \Z}{\forall}\ \|x|_{j}^{j+N-1} \|_0 \leq K\} \]
for $K,N \in \N,\ K \leq N$. We will show that for any $\alpha \in (0,1]$
\begin{equation}\label{eq:LIN-H_sparse_lower2} \sup \limits_{\mu \in \mP_{\sigma}(\mS)} \r_{\LIN - \mH_{\alpha}}(\mu, 0) \geq \min \Big\{\frac{2K}{N}, 1 \Big\} = \min \{ 2\mmdim(\mS), 1 \}.
\end{equation}
Actually, we will construct a single measure $\mu$ which realizes the above bound. Let $\nu \in \Prob([0,1]^N)$ be any Borel probability measure with $\supp(\nu) = [0,1]^K \times \{0 \}^{N-K}$ (e.g. product of the Lebesgue measure on $[0,1]^K$ with the Dirac's delta at zero in $[0,1]^{N-K}$). Define measure
\[\mu = \frac{1}{N} \sum \limits_{j=0}^{N-1} \sigma^j_*(\bigotimes \limits_{\Z} \nu). \]
Note that $\mu \in \mP_{\sigma}(\mS)$. Fix $n,k \in \N$ and consider any pair of functions $f : [0,1]^n \to [0,1]^k,\ g : [0,1]^k \to [0,1]^n$ such that $f \in \LIN,\ g \in \mH_{\alpha}$ and 
\begin{equation}\label{eq:LIN-H_sparse_full_meas} \mu(\{ x \in \mS : g \circ f (x|_0^{n-1}) \neq x|_0^{n-1} \}) = 0.
\end{equation} We will prove now a lower bound on $k$ and use it to conclude (\ref{eq:LIN-H_sparse_lower2}). We can decompose $n$ uniquely as $n = \ell N + m$ for some $\ell \in \N_0$ and $m \in \{ 0, 1, \ldots, N-1\}$. Define set $A \subset [0,1]^n$ as
\[ A =  \Big(\big([0,1]^K \times \{0 \}^{N-K}\big)^{\ell} \times \{0\}^m \Big) \cup \Big( \big(\{0 \}^{N-K} \times [0,1]^K\big)^{\ell} \times \{0\}^m \Big).\]
Note that $A \subset \supp((\pi_n)_*\mu)$. As both $f$ and $g$ are continuous, the set $\{ x \in [0,1]^n : g \circ f (x) = x \}$ is closed and, by (\ref{eq:LIN-H_sparse_full_meas}), of full measure $(\pi_n)_*\mu$. Therefore $\supp((\pi_n)_*\mu) \subset \{ x \in [0,1]^n : g \circ f (x) = x \}$ and consequently
\[ A \subset \{ x \in [0,1]^n : g \circ f (x) = x \}.\]
Therefore $f$ is injective on $A$. Let $F: \R^n \to \R^k$ be the linear extension of $f$ (it exists as $f \in \LIN$). We claim that $F$ is injective on a larger set
\[ E =  \Big(\big(\R^K \times \{0 \}^{N-K}\big)^{\ell} \times \{0\}^m \Big) \cup \Big( \big(\{0 \}^{N-K} \times \R^K\big)^{\ell} \times \{0\}^m \Big).\]
Indeed, if $x,y \in E$ are such that $F(x)=F(y)$, then there exists $v \in \R^n$ and $t>0$ such that $t(x + v) \in A,\ t(y + v)\in A$ and $f(t(x + v)) = f(t(y + v))$. Consequently $t(x + v) = t(y + v)$, hence $x=y$. As $F$ is linear, its injectivity on $E$ implies that 
\begin{equation}\label{eq:LIN-H_sparse_ker}
\Ker(F) \cap (E - E) = \{ 0 \}
\end{equation}
(here $E - E$ denotes the algebraic difference, i.e. $E - E = \{ x - y : x,y\in E\}$). Note that $E-E$ contains a linear subspace of dimension $\min\{2\ell K, \ell N\}$. More precisely:
\[ \big(\R^K \times \{0 \}^{N-2K} \times \R^K\big)^{\ell} \times \{0\}^m \subset E - E\ \text{  if  } 2K \leq N\]
and
\[ \big(\R^N\big)^{\ell} \times \{0\}^m \subset E - E\ \text{  if  }  2K > N. \]
Therefore (\ref{eq:LIN-H_sparse_ker}) implies $\dim(\Ker(F)) \leq n - \min\{2\ell K, \ell N\}$ and consequently $k \geq \dim(\mathrm{Im}(F)) \geq \min\{2\ell K, \ell N \}$. This gives bound
\[ \r_{\LIN - \mH^p_{L,\alpha}} (\mu, 0, n) \geq \frac{\min\{2\ell K, \ell N \}}{n} = \frac{\ell\min\{2K, N\}}{\ell N + m} \geq \frac{\ell\min\{2K, N\}}{(\ell+1)N} = \frac{\ell}{\ell+1} \min \Big\{\frac{2K}{N}, 1 \Big\}. \]
Taking $\limsup \limits_{n \to \infty}$ and recalling that $l \to \infty$ as $n \to \infty$, we obtain finally
\[ \r_{\LIN - \mH^p_{L,\alpha}} (\mu, 0) \geq \min \Big\{\frac{2K}{N}, 1 \Big\}. \]
Equality $\min \Big\{\frac{2K}{N}, 1 \Big\} = \min \{ 2\mmdim(\mS), 1 \}$ follows from the calculation in Example \ref{ex:sparse_subshift}. This concludes the proof of (\ref{eq:LIN-H_sparse_lower2}). Note that the upper bound $\frac{2}{1-\alpha}\mmdim(\mS)$ of Theorem \ref{thm:main_holder_up} approaches the above lower bound $2\mmdim(\mS)$ as $\alpha \to 0$, hence the constant $\frac{2}{1-\alpha}$ in Theorem \ref{thm:main_holder_up} cannot be in general replaced by any smaller one.

\section{Lower bound for $\r_{\mathcal{B}-\mathcal{H}_{\alpha}}(\mu,\eps)$}\label{app:wv_lower_bound}
Following closely the proof of \cite[Lemma 13]{WV10} (see also \cite[Equation (75)]{WV10})
we have the following proposition:
\begin{prop}\label{prop:rbh_below}
	Let $\mS$ be a closed and shift-invariant subspace of \textup{$[0,1]^{\mathbb{Z}}$}
	and $\mu\in  \mP_{\sigma}(\mS)$. Then for $0<\eps<1$ and $\alpha \in (0,1]$ the following holds:
	
	\[
	\alpha R_{B}(\mu,\eps)\leq \r_{\mathcal{B}-\mathcal{H_{\alpha}}}(\mu,\eps).
	\]
\end{prop}

\begin{proof}
	Fix $n\in\N$ and let $\r_{\mathcal{B}-\mathcal{H}_{\alpha}}(\mu,\eps,n)=\frac{k}{n}$. There exist
	$f:[0,1]^n\to[0,1]^{k},\ g:[0,1]^{k}\to[0,1]^{n}$ such
	that $f$ is Borel, $g$ is $(L,\alpha)$-H\"older and $B=\{x\in \mS:g\circ f(x|_{0}^{n-1})\neq x|_{0}^{n-1}\}$
	satisfies $\mu(B)\leq\eps$. Take $A=\pi_{n}(\mS\setminus B)\subset[0,1]^{n}$. Then $A=g(f(A)),\ f(A) \subset [0,1]^{k}$ and $\udim(f(A)) \leq k$. Since $g$ is $\alpha$-H\"{o}lder, we have by \cite[Lemma 3.3.(iv)]{Rob11} that $\udim(A) = \udim(g(f(A))) \leq \frac{1}{\alpha} \udim(f(A)) \leq \frac{k}{\alpha}$. We also have $\mu(\pi_{n}^{-1}(A))\geq\mu(\mS\setminus B)\geq1-\eps$,
	hence $R_{B}(\mu,\eps)\leq\frac{1}{\alpha}\frac{k}{n}=\frac{1}{\alpha}\r_{\mathcal{B}-\mathcal{{H}}_{\alpha}}(\mu,\eps,n)$.
	Taking $\limsup\limits_{n\to\infty}$ on the right side, we get the
	desired result.
\end{proof}

\section*{Acknowledgments}
We are grateful to Amos Lapidoth, Neri Merhav and Erwin Riegler for helpful discussions. We are grateful to the editor Ioannis Kontoyiannis and the anonymous reviewers for very helpful comments.

\def\cprime{$'$} \def\cprime{$'$}


\begin{thebibliography}{10}
\providecommand{\url}[1]{#1}
\csname url@samestyle\endcsname
\providecommand{\newblock}{\relax}
\providecommand{\bibinfo}[2]{#2}
\providecommand{\BIBentrySTDinterwordspacing}{\spaceskip=0pt\relax}
\providecommand{\BIBentryALTinterwordstretchfactor}{4}
\providecommand{\BIBentryALTinterwordspacing}{\spaceskip=\fontdimen2\font plus
\BIBentryALTinterwordstretchfactor\fontdimen3\font minus
  \fontdimen4\font\relax}
\providecommand{\BIBforeignlanguage}[2]{{%
\expandafter\ifx\csname l@#1\endcsname\relax
\typeout{** WARNING: IEEEtran.bst: No hyphenation pattern has been}%
\typeout{** loaded for the language `#1'. Using the pattern for}%
\typeout{** the default language instead.}%
\else
\language=\csname l@#1\endcsname
\fi
#2}}
\providecommand{\BIBdecl}{\relax}
\BIBdecl

\bibitem{GS19short}
Y.~{Gutman} and A.~{\'Spiewak}, ``New uniform bounds for almost lossless analog
  compression,'' in \emph{2019 IEEE International Symposium on Information
  Theory (ISIT)}, 2019, pp. 1702--1706.

\bibitem{WV10}
\BIBentryALTinterwordspacing
Y.~Wu and S.~Verd{\'u}, ``R\'enyi information dimension: fundamental limits of
  almost lossless analog compression,'' \emph{IEEE Trans. Inform. Theory},
  vol.~56, no.~8, pp. 3721--3748, 2010. [Online]. Available:
  \url{http://dx.doi.org/10.1109/TIT.2010.2050803}
\BIBentrySTDinterwordspacing

\bibitem{WV12}
\BIBentryALTinterwordspacing
Y.~Wu and S.~Verd\'{u}, ``Optimal phase transitions in compressed sensing,''
  \emph{IEEE Trans. Inform. Theory}, vol.~58, no.~10, pp. 6241--6263, 2012.
  [Online]. Available: \url{https://doi.org/10.1109/TIT.2012.2205894}
\BIBentrySTDinterwordspacing

\bibitem{JP16}
\BIBentryALTinterwordspacing
S.~Jalali and H.~V. Poor, ``Universal compressed sensing,'' \emph{2016 IEEE
  International Symposium on Information Theory (ISIT)}, 2016. [Online].
  Available: \url{https://doi.org/10.1109/ISIT.2016.7541723}
\BIBentrySTDinterwordspacing

\bibitem{jalali2017universal}
------, ``Universal compressed sensing for almost lossless recovery,''
  \emph{IEEE Transactions on Information Theory}, vol.~63, no.~5, pp.
  2933--2953, 2017.

\bibitem{CompressionCS17}
\BIBentryALTinterwordspacing
F.~E. Rezagah, S.~Jalali, E.~Erkip, and H.~V. Poor, ``Compression-based
  compressed sensing,'' \emph{IEEE Trans. Inform. Theory}, vol.~63, no.~10, pp.
  6735--6752, 2017. [Online]. Available:
  \url{https://doi.org/10.1109/TIT.2017.2726549}
\BIBentrySTDinterwordspacing

\bibitem{stotz2013almost}
D.~Stotz, E.~Riegler, and H.~B{\"o}lcskei, ``Almost lossless analog signal
  separation,'' in \emph{2013 IEEE International Symposium on Information
  Theory Proceedings (ISIT)}.\hskip 1em plus 0.5em minus 0.4em\relax IEEE,
  2013, pp. 106--110.

\bibitem{SRAB17}
\BIBentryALTinterwordspacing
D.~Stotz, E.~Riegler, E.~Agustsson, and H.~B\"olcskei, ``Almost lossless analog
  signal separation and probabilistic uncertainty relations,'' \emph{IEEE
  Trans. Inform. Theory}, vol.~63, no.~9, pp. 5445--5460, 2017. [Online].
  Available: \url{https://doi.org/10.1109/TIT.2017.2711041}
\BIBentrySTDinterwordspacing

\bibitem{GK18}
B.~C. Geiger and T.~Koch, ``On the information dimension of stochastic
  processes,'' Preprint. https://arxiv.org/abs/1702.00645, 2018.

\bibitem{CT05}
\BIBentryALTinterwordspacing
E.~J. Cand\`{e}s and T.~Tao, ``Decoding by linear programming,'' \emph{IEEE
  Trans. Inform. Theory}, vol.~51, no.~12, pp. 4203--4215, 2005. [Online].
  Available: \url{https://doi.org/10.1109/TIT.2005.858979}
\BIBentrySTDinterwordspacing

\bibitem{CRT06}
\BIBentryALTinterwordspacing
E.~J. Cand\`es, J.~Romberg, and T.~Tao, ``Robust uncertainty principles: exact
  signal reconstruction from highly incomplete frequency information,''
  \emph{IEEE Trans. Inform. Theory}, vol.~52, no.~2, pp. 489--509, 2006.
  [Online]. Available: \url{https://doi.org/10.1109/TIT.2005.862083}
\BIBentrySTDinterwordspacing

\bibitem{Donoho06}
\BIBentryALTinterwordspacing
D.~L. Donoho, ``Compressed sensing,'' \emph{IEEE Trans. Inform. Theory},
  vol.~52, no.~4, pp. 1289--1306, 2006. [Online]. Available:
  \url{https://doi.org/10.1109/TIT.2006.871582}
\BIBentrySTDinterwordspacing

\bibitem{FR13}
\BIBentryALTinterwordspacing
S.~Foucart and H.~Rauhut, \emph{A mathematical introduction to compressive
  sensing}, ser. Applied and Numerical Harmonic Analysis.\hskip 1em plus 0.5em
  minus 0.4em\relax Birkh\"auser/Springer, New York, 2013. [Online]. Available:
  \url{https://doi.org/10.1007/978-0-8176-4948-7}
\BIBentrySTDinterwordspacing

\bibitem{eldar2015sampling}
Y.~C. Eldar, \emph{Sampling Theory: Beyond Bandlimited Systems}.\hskip 1em plus
  0.5em minus 0.4em\relax Cambridge University Press, 2015.

\bibitem{BD09}
\BIBentryALTinterwordspacing
T.~Blumensath and M.~E. Davies, ``Sampling theorems for signals from the union
  of finite-dimensional linear subspaces,'' \emph{IEEE Trans. Inform. Theory},
  vol.~55, no.~4, pp. 1872--1882, 2009. [Online]. Available:
  \url{https://doi.org/10.1109/TIT.2009.2013003}
\BIBentrySTDinterwordspacing

\bibitem{Blumensath11}
\BIBentryALTinterwordspacing
T.~Blumensath, ``Sampling and reconstructing signals from a union of linear
  subspaces,'' \emph{IEEE Trans. Inform. Theory}, vol.~57, no.~7, pp.
  4660--4671, 2011. [Online]. Available:
  \url{https://doi.org/10.1109/TIT.2011.2146550}
\BIBentrySTDinterwordspacing

\bibitem{LempelZiv78}
\BIBentryALTinterwordspacing
J.~Ziv and A.~Lempel, ``Compression of individual sequences via variable-rate
  coding,'' \emph{IEEE Trans. Inform. Theory}, vol.~24, no.~5, pp. 530--536,
  1978. [Online]. Available: \url{https://doi.org/10.1109/TIT.1978.1055934}
\BIBentrySTDinterwordspacing

\bibitem{Z78}
J.~{Ziv}, ``\BIBforeignlanguage{English}{{Coding theorems for individual
  sequences.}}'' \emph{\BIBforeignlanguage{English}{{IEEE Trans. Inf.
  Theory}}}, vol.~24, pp. 405--412, 1978.

\bibitem{KawanYuksel18}
\BIBentryALTinterwordspacing
C.~Kawan and S.~Y\"{u}ksel, ``On optimal coding of non-linear dynamical
  systems,'' \emph{IEEE Trans. Inform. Theory}, vol.~64, no.~10, pp.
  6816--6829, 2018. [Online]. Available:
  \url{https://doi.org/10.1109/TIT.2018.2844211}
\BIBentrySTDinterwordspacing

\bibitem{KawanYukselStoch18}
------, ``Metric and topological entropy bounds for optimal coding of
  stochastic dynamical systems,'' Preprint. https://arxiv.org/abs/1612.00564,
  2018.

\bibitem{LinderZamir06}
\BIBentryALTinterwordspacing
T.~Linder and R.~Zamir, ``Causal coding of stationary sources and individual
  sequences with high resolution,'' \emph{IEEE Trans. Inform. Theory}, vol.~52,
  no.~2, pp. 662--680, 2006. [Online]. Available:
  \url{https://doi.org/10.1109/TIT.2005.862075}
\BIBentrySTDinterwordspacing

\bibitem{benartzi1993holder}
A.~Benartzi, A.~Eden, C.~Foias, and B.~Nicolaenko, ``H\"{o}lder continuity for
  the inverse of {M}a\~{n}\'{e}'s projection,'' \emph{Journal of Mathematical
  Analysis and Applications}, vol. 178, no.~1, pp. 22--29, 1993.

\bibitem{LW00}
\BIBentryALTinterwordspacing
E.~Lindenstrauss and B.~Weiss, ``Mean topological dimension,'' \emph{Israel J.
  Math.}, vol. 115, pp. 1--24, 2000. [Online]. Available:
  \url{https://doi.org/10.1007/BF02810577}
\BIBentrySTDinterwordspacing

\bibitem{GutTsu16}
Y.~Gutman and M.~Tsukamoto, ``Embedding minimal dynamical systems into
  {H}ilbert cubes,'' Preprint. http://arxiv.org/abs/1511.01802, 2015.

\bibitem{L99}
E.~Lindenstrauss, ``Mean dimension, small entropy factors and an embedding
  theorem,'' \emph{Inst. Hautes \'Etudes Sci. Publ. Math.}, vol.~89, no.~1, pp.
  227--262, 1999.

\bibitem{gutman2017application}
Y.~Gutman, Y.~Qiao, and M.~Tsukamoto, ``Application of signal analysis to the
  embedding problem of $\mathbb{Z}^{k}$-actions,'' \emph{Preprint.
  https://arxiv.org/abs/1709.00125}, 2017.

\bibitem{lindenstrauss_tsukamoto2017rate}
\BIBentryALTinterwordspacing
E.~Lindenstrauss and M.~Tsukamoto, ``From rate distortion theory to metric mean
  dimension: variational principle,'' \emph{IEEE Trans. Inform. Theory},
  vol.~64, no.~5, pp. 3590--3609, 2018. [Online]. Available:
  \url{https://doi.org/10.1109/TIT.2018.2806219}
\BIBentrySTDinterwordspacing

\bibitem{Gray11}
\BIBentryALTinterwordspacing
R.~M. Gray, \emph{Entropy and information theory}, 2nd~ed.\hskip 1em plus 0.5em
  minus 0.4em\relax Springer, New York, 2011. [Online]. Available:
  \url{https://doi.org/10.1007/978-1-4419-7970-4}
\BIBentrySTDinterwordspacing

\bibitem{G}
M.~Gromov, ``Topological invariants of dynamical systems and spaces of
  holomorphic maps. {I},'' \emph{Math. Phys. Anal. Geom.}, vol.~2, no.~4, pp.
  323--415, 1999.

\bibitem{W82}
P.~Walters, \emph{An introduction to ergodic theory}, ser. Graduate Texts in
  Mathematics.\hskip 1em plus 0.5em minus 0.4em\relax New York:
  Springer-Verlag, 1982, vol.~79.

\bibitem{R87}
W.~Rudin, \emph{Real and complex analysis}, 3rd~ed.\hskip 1em plus 0.5em minus
  0.4em\relax McGraw-Hill Book Co., New York, 1987.

\bibitem{coornaert2015topological}
M.~Coornaert, \emph{Topological dimension and dynamical systems}.\hskip 1em
  plus 0.5em minus 0.4em\relax Springer, 2015.

\bibitem{Banach51}
S.~Banach, \emph{Wst\k{e}p do teorii funkcji rzeczywistych (Polish)
  [Introduction to the theory of real functions]}, ser. Monografie
  Matematyczne. Tom XVII.].\hskip 1em plus 0.5em minus 0.4em\relax Polskie
  Towarzystwo Matematyczne, Warszawa-Wroc\l{}aw, 1951.

\bibitem{minty1970extension}
G.~J. Minty, ``On the extension of lipschitz, lipschitz-h{\"o}lder continuous,
  and monotone functions,'' \emph{Bulletin of the American Mathematical
  Society}, vol.~76, no.~2, pp. 334--339, 1970.

\bibitem{SchwartzNFA}
J.~T. Schwartz, \emph{Nonlinear functional analysis}.\hskip 1em plus 0.5em
  minus 0.4em\relax Gordon and Breach Science Publishers, New
  York-London-Paris, 1969, notes by H. Fattorini, R. Nirenberg and H. Porta,
  with an additional chapter by Hermann Karcher, Notes on Mathematics and its
  Applications.

\bibitem{Rob11}
J.~C. Robinson, \emph{Dimensions, embeddings, and attractors}, ser. Cambridge
  Tracts in Mathematics.\hskip 1em plus 0.5em minus 0.4em\relax Cambridge:
  Cambridge University Press, 2011, vol. 186.

\bibitem{falconer2004fractal}
K.~Falconer, \emph{Fractal geometry: mathematical foundations and
  applications}.\hskip 1em plus 0.5em minus 0.4em\relax John Wiley \& Sons,
  2004.

\bibitem{mattila}
\BIBentryALTinterwordspacing
P.~Mattila, \emph{Geometry of sets and measures in {E}uclidean spaces}, ser.
  Cambridge Studies in Advanced Mathematics.\hskip 1em plus 0.5em minus
  0.4em\relax Cambridge University Press, Cambridge, 1995, vol.~44, fractals
  and rectifiability. [Online]. Available:
  \url{https://doi.org/10.1017/CBO9780511623813}
\BIBentrySTDinterwordspacing

\bibitem{velozo2017rate}
A.~Velozo and R.~Velozo, ``Rate distortion theory, metric mean dimension and
  measure theoretic entropy,'' \emph{arXiv preprint arXiv:1707.05762}, 2017.

\bibitem{Gallager68}
R.~G. Gallager, \emph{Information Theory and Reliable Communication}.\hskip 1em
  plus 0.5em minus 0.4em\relax John Wiley and Sons, 1968.

\bibitem{GK17}
B.~C. Geiger and T.~Koch, ``On the information dimension rate of stochastic
  processes,'' in \emph{2017 IEEE International Symposium on Information Theory
  (ISIT)}, June 2017, pp. 888--892.

\bibitem{HK99}
\BIBentryALTinterwordspacing
B.~R. Hunt and V.~Y. Kaloshin, ``Regularity of embeddings of
  infinite-dimensional fractal sets into finite-dimensional spaces,''
  \emph{Nonlinearity}, vol.~12, no.~5, pp. 1263--1275, 1999. [Online].
  Available: \url{http://dx.doi.org/10.1088/0951-7715/12/5/303}
\BIBentrySTDinterwordspacing

\bibitem{BGS18}
K.~Bara\'{n}ski, Y.~Gutman, and A.~\'{S}piewak, ``A probabilistic {T}akens
  theorem,'' Preprint. https://arxiv.org/abs/1811.05959, 2018.

\bibitem{LAC}
\BIBentryALTinterwordspacing
G.~Alberti, H.~B\"{o}lcskei, C.~De~Lellis, G.~Koliander, and E.~Riegler,
  ``Lossless analog compression,'' \emph{IEEE Trans. Inform. Theory}, vol.~65,
  no.~11, pp. 7480--7513, 2019. [Online]. Available:
  \url{https://doi.org/10.1109/TIT.2019.2923091}
\BIBentrySTDinterwordspacing

\bibitem{Milne80}
\BIBentryALTinterwordspacing
S.~C. Milne, ``Peano curves and smoothness of functions,'' \emph{Adv. in
  Math.}, vol.~35, no.~2, pp. 129--157, 1980. [Online]. Available:
  \url{https://doi.org/10.1016/0001-8708(80)90045-6}
\BIBentrySTDinterwordspacing

\bibitem{Sagan94}
\BIBentryALTinterwordspacing
H.~Sagan, \emph{Space-filling curves}, ser. Universitext.\hskip 1em plus 0.5em
  minus 0.4em\relax Springer-Verlag, New York, 1994. [Online]. Available:
  \url{https://doi.org/10.1007/978-1-4612-0871-6}
\BIBentrySTDinterwordspacing

\bibitem{ArnoldProblems}
V.~I. Arnold, \emph{Arnold's problems}.\hskip 1em plus 0.5em minus 0.4em\relax
  Springer-Verlag, Berlin; PHASIS, Moscow, 2004, translated and revised edition
  of the 2000 Russian original, With a preface by V. Philippov, A. Yakivchik
  and M. Peters.

\bibitem{Shchepin2010}
\BIBentryALTinterwordspacing
E.~V. Shchepin, ``On {H}\"{o}lder maps of cubes,'' \emph{Math. Notes}, vol.~87,
  no.~5, pp. 757--767, 2010. [Online]. Available:
  \url{https://doi.org/10.1134/S0001434610050135}
\BIBentrySTDinterwordspacing

\bibitem{K95}
A.~S. Kechris, \emph{Classical descriptive set theory}, ser. Graduate Texts in
  Mathematics.\hskip 1em plus 0.5em minus 0.4em\relax New York:
  Springer-Verlag, 1995, vol. 156.

\bibitem{LT12}
E.~Lindenstrauss and M.~Tsukamoto, ``Mean dimension and an embedding problem:
  an example,'' \emph{Israel J. Math.}, vol. 199, pp. 573--584, 2014.

\bibitem{GutLinTsu15}
Y.~Gutman, E.~Lindenstrauss, and M.~Tsukamoto, ``Mean dimension of
  $\mathbb{Z}^k$-actions,'' \emph{Geom. Funct. Anal.}, vol.~26, no.~3, pp.
  778--817, 2016.

\bibitem{TsBrody18}
\BIBentryALTinterwordspacing
M.~Tsukamoto, ``Mean dimension of the dynamical system of {B}rody curves,''
  \emph{Invent. Math.}, vol. 211, no.~3, pp. 935--968, 2018. [Online].
  Available: \url{https://doi.org/10.1007/s00222-017-0758-9}
\BIBentrySTDinterwordspacing

\bibitem{TsYangMills18}
\BIBentryALTinterwordspacing
------, ``Large dynamics of {Y}ang-{M}ills theory: mean dimension formula,''
  \emph{J. Anal. Math.}, vol. 134, no.~2, pp. 455--499, 2018. [Online].
  Available: \url{https://doi.org/10.1007/s11854-018-0014-2}
\BIBentrySTDinterwordspacing

\bibitem{CDZ17}
E.~Chen, D.~Dou, and D.~Zheng, ``Variational principles for amenable metric
  mean dimensions,'' Preprint. https://arxiv.org/abs/1708.02087, 2017.

\bibitem{BillingsleyPM}
P.~Billingsley, \emph{Probability and measure}, 3rd~ed., ser. Wiley Series in
  Probability and Mathematical Statistics.\hskip 1em plus 0.5em minus
  0.4em\relax John Wiley \& Sons, Inc., New York, 1995, a Wiley-Interscience
  Publication.

\end{thebibliography}
\end{document}